\documentclass[11pt]{amsart}
\usepackage{amssymb}
\usepackage{latexsym} 
\usepackage{srcltx} 
\usepackage{esint}
\usepackage{xcolor}
\newcommand{\ep}{\varepsilon}

\def\nc{\newcommand}

 \def\Om{\Omega}

\def\R{\mathbf R}
\nc\pa{\partial}

\nc\CC{\mathbb{C}}
\nc\RR{\mathbb{R}}
\nc\QQ{\mathbb{Q}}
\nc\ZZ{\mathbb{Z}}
\nc\NN{\mathbb{N}}
\nc{\defeq}{\mathrel{\mathop:}=}
\nc\MM{\mathcal{M}^{+}}
\nc\M{\mathcal{M}}
\nc\capa{\mathrm{cap}}
\nc\m[1]{\left| #1\right|}
\nc\norm[1]{\left\| #1\right\|}
\def\nc{\newcommand}

\def\ep{\epsilon}

 \def\Om{\Omega}

\def\R{\mathbf R}

\def\bea{\begin{equation}\begin{aligned}}
\def\ena{\end{aligned}\end{equation}}

\def\beas{\begin{equation*}\begin{aligned}}
\def\enas{\end{aligned}\end{equation*}}

\nc\axgrad[1]{\mathcal{A}(x, #1)}

\newtheorem{theorem}{Theorem}[section]
\newtheorem{lemma}[theorem]{Lemma}
\newtheorem{corollary}[theorem]{Corollary}
\newtheorem{proposition}[theorem]{Proposition}
\newtheorem{definition}[theorem]{Definition}
\newtheorem{remark}[theorem]{Remark}        

\numberwithin{equation}{section}

\begin{document}

\title[Uniqueness of entire solutions]{Uniqueness of entire solutions to quasilinear equations of $p$-Laplace type}


\author[Nguyen Cong Phuc]
{Nguyen Cong Phuc$^*$}
\address{Department of Mathematics,
Louisiana State University,
303 Lockett Hall, Baton Rouge, LA 70803, USA.}
\email{pcnguyen@math.lsu.edu}

\author[Igor E. Verbitsky]
{Igor E. Verbitsky}
\address{Department of Mathematics,
	 University of Missouri,
	Columbia, MO 65211, USA.}
\email{verbitskyi@missouri.edu}

\thanks{$^*$Supported in part by Simons Foundation, award number 426071}


\begin{abstract} 
We prove the uniqueness property for a class of entire solutions to the equation 
\begin{equation*}
\left\{ \begin{array}{ll}
-{\rm div}\, \mathcal{A}(x,\nabla u) = \sigma, \quad u\geq 0  \quad \text{in } \RR^n, \\
\displaystyle{\liminf_{|x|\rightarrow \infty}}\,  u = 0, 
\end{array}
\right.
\end{equation*}
where $\sigma$ is a nonnegative locally finite measure in $\RR^n$, absolutely continuous with respect 
to the $p$-capacity, and  
${\rm div}\, \mathcal{A}(x,\nabla u)$ is the $\mathcal{A}$-Laplace operator, under 
standard growth and monotonicity assumptions of order $p$ ($1<p<\infty$) on $\mathcal{A}(x, \xi)$ ($x, \xi \in \RR^n$); the model case 
$\mathcal{A}(x, \xi)=\xi | \xi |^{p-2}$  corresponds to the $p$-Laplace operator $\Delta_p$ on $\RR^n$. 

Our main results establish uniqueness of solutions to a similar problem,
\begin{equation*}
\left\{ \begin{array}{ll}
-{\rm div}\, \mathcal{A}(x,\nabla u) = \sigma u^q +\mu, \quad u\geq 0  \quad \text{in } \RR^n, \\
\displaystyle{\liminf_{|x|\rightarrow \infty}}\,  u = 0, 
\end{array}
\right.
\end{equation*}
 in the sub-natural growth case $0<q<p-1$, where 
$\mu, \sigma$ are nonnegative locally finite measures in 
$\RR^n$, absolutely continuous with respect 
to the $p$-capacity, and   $\mathcal{A}(x, \xi)$ satisfies an additional homogeneity condition, which holds in particular for 
the  $p$-Laplace operator. 
\end{abstract}


\maketitle

\section{Introduction}

We prove the uniqueness property for a class of \textit{reachable} solutions   to the equation 
\begin{equation}\label{p-Lapl} 
\left\{ \begin{array}{ll}
- \Delta_p u = \sigma, \quad u\geq 0  \quad \text{in } \RR^n, \\
\displaystyle{\liminf_{|x|\rightarrow \infty}}\,  u = 0, 
\end{array}
\right.
\end{equation}
where $\sigma\ge 0$ is a  \textit{locally finite} Borel measure in $\RR^n$ absolutely continuous with respect 
to the $p$-capacity, and  $\Delta_p u = {\rm div} (\nabla u | \nabla u |^{p-2})$ ($1<p<\infty$) is 
 the $p$-Laplace operator. 
 
 More general  $\mathcal{A}$-Laplace operators  
${\rm div}\, \mathcal{A}(x,\nabla u)$ in place of $\Delta_p$,  under 
standard growth and monotonicity assumptions of order $p$  on $\mathcal{A}(x, \xi)$ ($x, \xi \in \RR^n$), are treated as well (see Sec. \ref{A-superharm}). All solutions $u$ of  \eqref{p-Lapl} are understood  to be $\mathcal{A}$-superharmonic 
(or, equivalently, locally renormalized) solutions in $\RR^n$ (see  \cite{KKT} and Sec. \ref{Quasilinear} below). 

We often use bilateral global pointwise estimates of solutions to \eqref{p-Lapl} obtained by 
Kilpel\"ainen and Mal\'y  \cite{KM1, KM2} in terms of the Havin--Maz'ya--Wolff potentials (often called Wolff potentials) 
${\rm \bf W}_{1,\,p}\sigma$. Criteria of existence of solutions to \eqref{p-Lapl}, which ensure that ${\rm \bf W}_{1,\,p}\sigma\not\equiv \infty$, can be found in \cite{PV1} (see also Sec. \ref{Quasilinear} below).

In Sec. \ref{sub-natural}, we prove uniqueness of nontrivial \textit{reachable} solutions to the problem 
\begin{equation}\label{p-Lapl-mu} 
\left\{ \begin{array}{ll}
- \Delta_p u = \sigma u^q + \mu, \quad u\geq 0  \quad \text{in } \RR^n, \\
\displaystyle{\liminf_{|x|\rightarrow \infty}}\,  u = 0, 
\end{array}
\right.
\end{equation} in the sub-natural growth case $0<q<p-1$,
where $\mu, \sigma$ are nonnegative locally finite measures in $\RR^n$ absolutely continuous with respect 
to the $p$-capacity. We observe that such a uniqueness property generally fails in the case $q\ge p-1$.

When  we treat the uniqueness problem for  solutions  of equations of type \eqref{p-Lapl-mu} for $0<q<p-1$  with 
 more general  $\mathcal{A}$-Laplace operators   
${\rm div}\, \mathcal{A}(x,\nabla \cdot)$  in place of $\Delta_p$,  we impose the additional homogeneity condition 
$\mathcal{A}(x,\lambda \xi)=\lambda^{p-1} \mathcal{A}(x, \xi)$, for all $\xi \in\RR^n$ and $ \lambda>0$ (see Sec. \ref{sub-natural}). Our main tool 
 in the proof of uniqueness is provided by bilateral pointwise estimates for  
 all entire 
 solutions obtained recently in  \cite{Ver1}. 


We observe that in the case $p=2$ \textit{all} superharmonic solutions of  equations  \eqref{p-Lapl} and  \eqref{p-Lapl-mu}   are reachable, 
and hence unique. An analogue of this fact is  true  for more general equations with the  linear uniformly elliptic $\mathcal{A}$-Laplace operator   
${\rm div}\, (\mathcal{A}(x) \nabla u)$, with bounded measurable coefficients $\mathcal{A} \in L^\infty(\R^n)^{n\times n}$, in place of $\Delta$. In other words, all entire $\mathcal{A}$-superharmonic  solutions to such equations are unique.  For similar problems in domains $\Omega \subseteq \RR^n$ and linear operators with positive Green's function  satisfying some additional properties (in particular, in uniform domains) 
the uniqueness property  was obtained recently in \cite{Ver2}. 

The uniqueness of nontrivial \textit{bounded} (superharmonic) solutions for \eqref{p-Lapl-mu} in the case $p=2$ was proved earlier by Brezis and Kamin \cite{BK}. 
For  solutions $u\in C(\overline{\Omega})$ in 
bounded smooth domains $\Omega\subset \RR^n$ and $\mu, \sigma \in C(\overline{\Omega})$, 
along with some more general equations involving monotone increasing, concave nonlinearities on the right-hand side, the uniqueness 
property was originally established by Krasnoselskii \cite[Theorem 7.14]{Kra}.

As shown below, for $p\not=2$, \textit{all} $p$-superharmonic  solutions $u$ to  \eqref{p-Lapl} or \eqref{p-Lapl-mu} are reachable, 
and hence unique,  if, for instance, the condition $\displaystyle{\liminf_{|x|\rightarrow \infty}}\,  u = 0$ in 
\eqref{p-Lapl} or \eqref{p-Lapl-mu}, respectively, is replaced with  
$\displaystyle{\lim_{|x|\rightarrow \infty}}\,  u = 0$. See Sections \ref{Quasilinear} and   \ref{sub-natural},  where we discuss this and other conditions that  ensure that 
all solutions are reachable. 

Existence  and bilateral pointwise estimates for all  $\mathcal{A}$-superharmonic solutions to \eqref{p-Lapl-mu}  were  
obtained in \cite{Ver1}. (See also earlier results in \cite{CV2}  involving \textit{minimal}  solutions in the case $\mu=0$.) In particular, it is known  that the measure $\sigma$ is necessarily absolutely continuous with respect 
to the $p$-capacity provided there exists a nontrivial  $u \ge 0$ such that $- \Delta_p u \ge  \sigma u^q$ (\cite{CV2}, Lemma 3.6).  

We remark that  the proofs of 
the main existence results in  \cite{CV2, Ver1} for \eqref{p-Lapl-mu}  in the case $\mu=0$ used a version of the comparison principle (\cite{CV2}, Lemma 5.2) that contained some inaccuracies.  A corrected form 
of this comparison principle is provided in  
 Lemma \ref{CP-corrected}  below. The other parts of  
 \cite{CV2, Ver1} are unaffected by this correction.  
 
 With regards to the existence problem,  we prove additionally that  we can always construct a reachable solution to either  \eqref{p-Lapl} or
 \eqref{p-Lapl-mu}, whenever a solution to the corresponding equation exists 
 (see Theorem \ref{EandU} and Remark \ref{rem-last}  below).

\section{$\mathcal{A}$-superharmonic functions}\label{A-superharm}

Let $\Om\subseteq \RR^n$, $n\geq 2$, be an open set. By $\MM(\Om)$ we denote the  cone of nonnegative locally finite Borel measures in $\Om$, and by  $\MM_b(\Om)$ the subcone of finite measures 
in $\MM(\Om)$. For $\mu \in \MM(\Om)$, we set $\Vert \mu\Vert_{\MM(\Omega)}=\mu(\Omega)$ even if 
$\mu(\Omega)=+\infty$. The space of finite  signed  Borel measures in $\Om$ is denoted by $\M_b(\Om)$. By $ \Vert \mu\Vert_{\M_b(\Om)}$ we denote the total variation of $\mu \in \M_b(\Om)$.

Let   $\mathcal{A}\colon  \RR^n \times \RR^n\rightarrow\RR^n$ be a Carath\'eodory function in the sense that 
the map $x \rightarrow \mathcal{A}(x,\xi)$  is  measurable~ for~ all~ $\xi\in\RR^n,$ and 
the~ map ~ $\xi\rightarrow \mathcal{A}(x,\xi)$  ~is~ continuous~ for~ a.e. $x\in\RR^n$.
Throughout the paper, we assume that   there are constants $0<\alpha\leq\beta<\infty$ and $1<p<n$ such that for a.e. $x$ in $\RR^n$,
\begin{equation}\label{structure}
\begin{aligned}
& \mathcal{A}(x,\xi)\cdot\xi\geq \alpha|\xi|^p,\quad |\mathcal{A}(x,\xi)|\leq\beta|\xi|^{p-1},  \quad \forall\, \xi 
\in \RR^n, \\
& [\mathcal{A}(x,\xi_{1})-\mathcal{A}(x,\xi_{2})]\cdot(\xi_{1}-\xi_{2})>0, \quad \forall\, \xi_{1}, \xi_2 \in \RR^n, \, \, 
\xi_{1}\not = \xi_{2}.
\end{aligned}
\end{equation}

In  the uniqueness results of Sec. \ref{sub-natural}, we assume additionally the homogeneity condition 
\begin{equation}\label{homogeneity}
\mathcal{A}(x,\lambda \xi)=\lambda^{p-1} \mathcal{A}(x, \xi), \qquad \forall \, \xi 
\in \RR^n,  \,
\lambda>0.
\end{equation}
Such homogeneity conditions are often used in the literature (see \cite{HKM}, \cite{KM2}).

For an open set $\Om\subset\RR^n$, it is well known that every weak solution $u\in W^{1,\,p}_{{\rm loc}}(\Om)$  to the
equation
\begin{eqnarray}
\label{homo}
-\text{div}\mathcal{A}(x,\nabla u)=0 \qquad \textrm{in} \, \, \Om
\end{eqnarray}
has a continuous representative. Such continuous solutions are said to be
$\mathcal{A}$-$harmonic$ in $\Om$. If $u\in W_{{\rm loc}}^{1,\,p}(\Om)$ and
\begin{eqnarray*}
	\int_{\Om}\mathcal{A}(x,\nabla u)\cdot\nabla\varphi \, dx\geq 0,
\end{eqnarray*}
for all nonnegative $\varphi\in C^{\infty}_{0}(\Om)$, i.e., $-{\rm div}\mathcal{A}(x,\nabla u)
\geq 0$ in the distributional sense, then $u$ is called a {\it supersolution}
to (\ref{homo}) in $\Om$.\\
\indent A  function $u\colon\Om\rightarrow (-\infty, \infty]$ is called
$\mathcal{A}$-$superharmonic$ if $u$ is not identically infinite in each connected component
of $\Om$, $u$ is lower semicontinuous, and for all open sets  $D$ such that
${\overline D}\subset\Om$, and all functions $h\in C(\overline{D})$, $\mathcal{A}$-harmonic in $D$, it follows that
$h\leq u$ on $\partial D$ implies $h\leq u$ in $D$.

A typical example of $\mathcal{A}(x,\xi)$ is given by $\mathcal{A}(x,\xi)=|\xi|^{p-2}\xi$, which gives rise to the $p$-Laplacian 
$\Delta_p u= {\rm div}\, (|\nabla u|^{p-2}\nabla u)$. In this case, $\mathcal{A}$-superharmonic functions will be called $p$-superharmonic functions.

  We recall here the fundamental connection between supersolutions of (\ref{homo}) and
$\mathcal{A}$-superharmonic functions discussed in \cite{HKM}.
\begin{proposition}[\cite{HKM}] \label{pro2.1}{\rm (i)} If $u\in W_{{\rm loc}}^{1,\,p}(\Om)$
	is such that
	\begin{eqnarray*}
		-{\rm div}\mathcal{A}(x,\nabla u)\geq 0 \qquad \textrm{in} \, \, \Om,
	\end{eqnarray*}
	then there is an $\mathcal{A}$-superharmonic function $v$ such that $u=v$ a.e. Moreover,
	\begin{eqnarray}
	\label{liminf}
	v(x)={\rm ess}\liminf_{y\rightarrow x}v(y), \qquad  x\in\Om.
	\end{eqnarray}
	\indent {\rm (ii)} If v is $\mathcal{A}$-superharmonic, then (\ref{liminf}) holds. Moreover,
	if $v\in W^{1,\,p}_{{\rm loc}}(\Om)$, then
	\begin{eqnarray*}
		-{\rm div}\mathcal{A}(x, \nabla v)\geq 0  \qquad \textrm{in} \, \, \Om.
	\end{eqnarray*}
	\indent {\rm (iii)} If v is $\mathcal{A}$-superharmonic and locally bounded, then
	$v\in W^{1,\, p}_{{\rm loc}}(\Om)$,    and
	\begin{eqnarray*}
		-{\rm div}\mathcal{A}(x, \nabla v)\geq 0  \qquad \textrm{in} \, \, \Om.
	\end{eqnarray*}
\end{proposition}

Note that if $u$ is  $\mathcal{A}$-superharmonic, then the gradient of $u$ may not exist in the sense of distributions in the case $1<p\leq 2-1/n$.
On the other hand, if  $u$ is an $\mathcal{A}$-superharmonic function, then 
  its truncation $u_k=\min\{u,k\}$ is $\mathcal{A}$-superharmonic as well, for any
  $k>0$. Moreover, by Proposition \ref{pro2.1}(iii) we have  $u_k\in W^{1,\,p}_{{\rm loc}}(\Om)$. Using this we define the very weak gradient
\begin{eqnarray*}
	Du \defeq \lim_{k\rightarrow\infty} \, \nabla  [ \, \min\{u,k\}] \qquad {\rm a.e.}  \,\, \textrm{in} \, \, \Om.
\end{eqnarray*}
If either $u\in L^{\infty}(\Om)$ or $u\in W^{1,\,1}_{{\rm loc}}(\Om)$, then $Du$
coincides with the regular distributional gradient of $u$. In general we have the following
gradient estimates \cite{KM1} (see also \cite{HKM}).
\begin{proposition}[\cite{KM1}]\label{gradient} Suppose u is $\mathcal{A}$-superharmonic in $\Om$ and
	$1\leq q< \frac{n}{n-1}$. Then both $|Du|^{p-1}$ and $\mathcal{A}(\cdot,Du)$
	belong to $L^{q}_{{\rm loc}}(\Om)$. Moreover, if  $p>2-\frac{1}{n}$, then $Du$ coincides
with	the distributional gradient of u.
\end{proposition}

%
 Note that  by Proposition \ref{gradient} and the
dominated convergence theorem, we have 
\begin{align*}
	-{\rm div}\mathcal{A}(x, \nabla u)(\varphi)&:=\int_{\Om}\mathcal
	{A}(x, D u)\cdot \nabla \varphi \, dx\\
	& = \lim_{k\rightarrow\infty}\int_{\Om}\mathcal
	{A}(x, \nabla\min\{u,k\})\cdot \nabla \varphi \, dx \geq 0,
\end{align*}
whenever $\varphi\in C^{\infty}_{0}(\Om)$, $\varphi\geq 0$, and $u$ is $\mathcal{A}$-superharmonic in $\Om$.
It follows  from Riesz's representation theorem (see \cite[Theorem 21.2]{HKM})
that there exists a unique measure $\mu[u]\in \MM(\Om)$ called the Riesz measure of $u$ such that
\begin{eqnarray*}
	-{\rm div}\mathcal{A}(x, \nabla u)=\mu[u] \qquad \textrm{in} \, \, \Om.
\end{eqnarray*}

\section{Quasilinear equations with locally finite measure data in the entire space}\label{Quasilinear}

In this section, we investigate the problems of  existence and uniqueness  of $\mathcal{A}$-superharmonic solutions in the entire space $\RR^n$ to the equation
\begin{equation}\label{Basic-PDE}
\left\{ \begin{array}{ll}
-{\rm div}\, \mathcal{A}(x,\nabla u) = \sigma, \quad u\geq 0  \quad \text{in } \RR^n, \\
\displaystyle{\liminf_{|x|\rightarrow \infty}}\,  u = 0, 
\end{array}
\right.
\end{equation}
with measures $\sigma \in \MM(\RR^n)$ (not necessarily finite).

There has been a lot of work addressing  the   existence and uniqueness problem for quasilinear equations
of the form 
\begin{equation}\label{Om-equation}
\left\{ \begin{array}{ll}
-{\rm div}\, \mathcal{A}(x,\nabla u) = \sigma   \quad \text{in } \Omega, \\
  u = 0 \quad \text{on } \partial\Om, 
\end{array}
\right.
\end{equation}
in a \textit{bounded} domain $\Om\subset \RR^n$, 
where $\sigma\in L^1(\Om)$, or, more generally, $\sigma\in \M_b(\Om)$; see, e.g., \cite{BB6, BG1, BG2, BGO, Da, DMM, DMOP, MP, KX}. For arbitrary  domains, including $\RR^n$, we refer to the papers \cite{BB6} (for $L^1$ data) and  \cite{MP} (for data in  $\M_b(\Om)$).  In these papers one can find the notions of {\it entropy solutions} (see \cite{BB6, BGO, KX}), {\it SOLA} (solutions obtained as limits of approximations) for $L^1$ data (see \cite{Da}), {\it reachable solutions} (see \cite{DMM}), and {\it renormalized solutions} (see \cite{DMOP}).

The current state of the art on the uniqueness problem for \eqref{Om-equation} 
is that most results require that $\sigma<<{\rm cap}_p$, i.e., $\sigma$ is absolutely continuous with respect to the $p$-capacity in the sense that $\sigma(K)=0$ for any compact set  $K\subset \Omega$ such that ${\rm cap}_p(K)=0$.     The $p$-capacity ${\rm cap}_p(\cdot)$ is a natural capacity associated with the $p$-Laplacian  defined  by
$${\rm cap}_p(K)\defeq \inf\left\{\int_{\Omega} |\nabla h|^p dx\colon \, \,  h\in C_0^\infty(\Omega), \, \, h\geq 1 \text{ on } K \right\},$$
for any compact set  $K\subset \Omega$.

For later use, we now recall the following  equivalent  definitions of a (global) renormalized  solution to equation \eqref{Om-equation} (see \cite{DMOP}). For our purposes, we shall restrict ourselves to the case $\sigma\in \MM_b(\Om)$ and nonnegative solutions.    
Recall that we may use the decomposition   $\sigma=\sigma_0+\sigma_s$,  where both $\sigma_0$ and $\sigma_s$ are nonnegative measures such that 
$\sigma_0<< {\rm cap}_0$, and $\sigma_s$ is concentrated on a set of zero $p$-capacity.

\begin{definition}\label{rns1}
		Let $\sigma\in \MM_b(\Om)$, where  $\Om\subset\RR^n$ is  a bounded open set. Then $u\geq 0$ is said to be a renormalized
		solution of \eqref{Om-equation}
		if the following conditions hold:\\
		\noindent {\rm (a)} The function $u$ is measurable and finite almost everywhere, and $T_{k}(u)$ belongs
		to $W_{0}^{1,\,p}(\Om)$ for every $k>0$, where $T_{k}(s)\defeq\min
		\{k,s\}$, $s\ge 0$.\\
		\noindent {\rm (b)} The gradient $Du$ of $u$ satisfies $\m{Du}^{p-1}\in L^{q}(\Om)$ for all $q<\frac{n}{n-1}$.\\
		\noindent {\rm (c)} For any $h\in W^{1,\infty}(\RR)$ with compact support, and any $\varphi\in W^{1,p}(\Om)\cap L^\infty(\Om)$ such that 
		$h(u) \, \varphi \in W^{1,p}_0(\Om)$,
		$$\int_{\Om}\mathcal{A}(x, D u)\cdot \nabla (h(u) \, \varphi) \, dx=\int_{\Om} h(u) \,  \varphi \, d\sigma_0,$$
		and for any $\varphi\in C^{0}_{\rm b}(\Om)$ (the space of bounded and continuous functions in $\Om$),
		$$\lim_{m\rightarrow\infty} \frac{1}{m}\int_{\{ m\leq u\leq 2m\}} 
		\mathcal{A}(x, D u)\cdot D u \,\varphi \, dx =\int_{\Om} \varphi \, d\sigma_s. $$ 
		\end{definition}
	
\begin{definition}\label{hprimesupp}
Let $\sigma\in \MM_b(\Om)$, where $\Om\subset\RR^n$ is a bounded open set.  Then  $u\geq 0$ is said to be a renormalized
solution of \eqref{Om-equation}
if $u$ satisfies {\rm (a)} and {\rm (b)} in Definition \ref{rns1}, and if the following condition holds:\\
\noindent {\rm (c)} For any $h\in W^{1,\infty}(\RR)$ with such that $h'$ has compact support, and any $\varphi\in W^{1,r}(\Om)\cap L^\infty(\Om)$, $r>n$, such that 
$h(u)\varphi \in W^{1,p}_0(\Om)$,
$$\int_{\Om}\mathcal{A}(x, D u)\cdot \nabla (h(u) \, \varphi) \, dx=\int_{\Om} h(u) \, \varphi \, d\sigma_0 +  h(+\infty)\int_{\Om} \varphi \, d \sigma_s.$$
Here $h(+\infty)\defeq\lim_{s\rightarrow +\infty} h(s)$.
\end{definition}	
	
\begin{definition}\label{defrs} 
		Let $\sigma\in \MM_b(\Om)$, where $\Om\subset\RR^n$ is a bounded open set. Then  $u\geq 0$ is said to be a renormalized
		solution of \eqref{Om-equation}
		if $u$ satisfies {\rm (a)} and {\rm (b)} in Definition \ref{rns1}, and if the following conditions hold:\\
		\noindent {\rm (c)} For every $k>0$, there exists $\lambda_k \in \MM_b(\Om)$ concentrated on the set $\{u=k\}$ such that 
		 $\lambda_k<<{\rm cap}_p$, and 
		$\lambda_{k}\rightarrow \sigma_{s}$ in the narrow topology
		of measures in $\Om$ as $k\rightarrow\infty$, i.e.,
		 $$\lim_{k\rightarrow\infty} \int_{\Om} \varphi \, d\lambda_k= \int_{\Om} \varphi \, d\sigma_s, \quad \forall \varphi\in C_{\rm b}^0(\Om).$$
		 
		\noindent {\rm (d)} For every $k>0$, 
		\begin{equation*}
		\int_{\{u<k\}}\mathcal{A}(x, Du)\cdot\nabla\varphi \, dx= \int_{\{u<k\}} \varphi
		\, d \sigma_{0} +  \int_{\Om}\varphi \, d \lambda_{k} 
		\end{equation*}
		for all $\varphi \in {\rm W}_{0}^{1,\,p}(\Om)\cap L^{\infty}(\Om)$.
\end{definition}

We shall also need the notion of a \textit{local} renormalized (nonnegative) solution on a general open set $\Om\subseteq\RR^n$ (not necessarily bounded) associated with a measure $\sigma \in \MM(\Om)$ 
(not necessarily finite). We recall the following equivalent definitions  (see \cite{BV}), adapted to the case of nonnegative solutions.
\begin{definition}\label{rns1-local}
	Let $\sigma \in \MM(\Om)$, where   $\Om\subseteq\RR^n$ is an open set. Then a nonnegative function $u$ is said to be a local renormalized
	solution of the equation $-{\rm div} \mathcal{A}(x, \nabla u)=\sigma$,
	if the following conditions hold:\\
	\noindent {\rm (a)} The function $u$ is measurable and finite almost everywhere, and $T_{k}(u)$ belongs
	to $W_{\rm loc}^{1,\,p}(\Om)$, for every $k>0$, where $T_{k}(s)\defeq \min
	\{k,s\}$, $s \ge 0$.\\
	\noindent {\rm (b)} The gradient $Du$ of $u$ satisfies $\m{Du}^{p-1}\in L^{q}_{\rm loc}(\Om)$ for all $0<q<\frac{n}{n-1}$, and 
	$u^{p-1} \in L^s_{\rm loc}(\Om)$ for all $0<s<\frac{n}{n-p}$.\\
	\noindent {\rm (c)} For any $h\in W^{1,\infty}(\RR)$ with compact support, and any $\varphi\in W^{1,p}(\Om)\cap L^\infty(\Om)$ with compact support in $\Om$ such that 
	$h(u) \, \varphi \in W^{1,p}(\Om)$,
	$$\int_{\Om}\mathcal{A}(x, D u)\cdot \nabla (h(u) \, \varphi) \, dx=\int_{\Om} h(u) \, \varphi \, d\sigma_0,$$
	and for any $\varphi\in C_{\rm b}^{0}(\Om)$ with compact support in $\Om$,
	$$\lim_{m\rightarrow\infty} \frac{1}{m}\int_{\{ m\leq u\leq 2m\}} 
	\mathcal{A}(x, D u)\cdot Du 
	\, \varphi \, dx =\int_{\Om} \varphi \, d\sigma_s. $$ 
\end{definition}

\begin{definition}
		Let $\sigma \in \MM(\Om)$, where   $\Om\subseteq\RR^n$ is an open set. Then a nonnegative function $u$ is said to be a local renormalized
	solution of the equation $-{\rm div} \mathcal{A}(x, \nabla u)=\sigma$,
		if $u$ satisfies {\rm (a)} and {\rm (b)} in Definition \ref{rns1-local}, and if the following conditions hold:\\
	\noindent {\rm (c)} For every $k>0$, there exists a nonnegative measure $\lambda_k<<{\rm cap}_p$, concentrated on the sets $\{u=k\}$,  such that
	$\lambda_{k}\rightarrow \sigma_{s}$ weakly  as measures in $\Om$ as $k\rightarrow\infty$, i.e.,
	$$\lim_{k\rightarrow\infty} \int_{\Om} \varphi \, d\lambda_k= \int_{\Om} \varphi \, d\sigma_s,$$
	for all $\varphi\in C^0_{\rm b}(\Om)$ with compact support in $\Om$.\\

	\noindent {\rm (d)} For every $k>0$, 
	\begin{equation*}
	\int_{\{u<k\}} \mathcal{A}(x, Du)\cdot\nabla\varphi \, dx= \int_{\{u<k\}} \varphi \, 
	d\sigma_{0} +  \int_{\Om} \varphi \, d\lambda_{k} 
	\end{equation*}
	for all $\varphi$ in ${\rm W}_{0}^{1,\,p}(\Om)\cap L^{\infty}(\Om)$ with compact support in $\Om$.
\end{definition}

We now discuss solutions of  \eqref{Basic-PDE} 
for general measures $\sigma \in \MM(\RR^n)$. It is known that  a necessary and sufficient condition for  \eqref{Basic-PDE} to admit an $\mathcal{A}$-superharmonic  solution is the finiteness condition
\begin{equation}\label{finiteness}
\int_{1}^{\infty} \left(\frac{\sigma(B(0,\rho))}{\rho^{n-p}} \right)^{\frac{1}{p-1}}\frac{d\rho}{\rho}<+\infty;
\end{equation} 
(see, e.g., \cite{PV1, PV2}). Thus, it is possible to solve  \eqref{Basic-PDE} for a wide and optimal class of measures $\sigma$ satisfying \eqref{finiteness} 
that are not necessarily finite.

We mention that \eqref{finiteness}  is equivalent to the condition ${\bf W}_{1,p}\sigma(x)<+\infty$ for some $x\in\RR^n$ (or equivalently
quasi-everywhere in $\RR^n$ with respect to the $p$-capacity),
where 
$${\bf W}_{1,p}\sigma(x):=\int_0^\infty \left(\frac{\sigma(B(x,\rho))}{\rho^{n-p}} \right)^{\frac{1}{p-1}}\frac{d\rho}{\rho}$$
is the Havin--Maz'ya--Wolff potential of $\sigma$ (often called the Wolff potential); see \cite{HW, Maz}. 

By the fundamental result of Kilpel\"ainen and Mal\'y  \cite{KM1, KM2}, any $\mathcal{A}$-superharmonic solution $u$ to equation \eqref{Basic-PDE} satisfies the following global pointwise estimates, 
\begin{equation}\label{Wolffbound}
\frac{1}{K} \, {\rm \bf W}_{1,\,p}\sigma(x)\leq u(x)\leq K \, {\rm \bf W}_{1,\,p}\sigma(x), \qquad \forall x\in\RR^n,
\end{equation}
 where $K>0$ is a constant depending only on $n, p$ and the structural constants
$\alpha$ and $\beta$ in \eqref{structure}.

Our main goal here is to introduce a new notion of a solution to \eqref{Basic-PDE} so that existence is obtained under the natural growth condition 
\eqref{finiteness} for $\sigma$, and uniqueness is guaranteed as long as $\sigma<<{\rm cap}_p$ (see Definition \ref{reach} below).

We begin with the following result on the existence of a minimal solution to  \eqref{Basic-PDE} in  case the measure $\sigma$ is continuous with respect to the $p$-capacity.
\begin{theorem}\label{miniexist}
	Let  $\sigma \in \MM(\RR^n)$, where $\sigma<< {\rm cap}_p$. Suppose that 
	\eqref{finiteness} holds.
	Then there exists a  minimal $\mathcal{A}$-superharmonic solution to equation \eqref{Basic-PDE}.
\end{theorem}

\begin{proof} Condition \eqref{finiteness} implies that 
	$$\int_{1}^{\infty} \left(\frac{\sigma(B(x,\rho))}{\rho^{n-p}} \right)^{\frac{1}{p-1}}\frac{d\rho}{\rho}<+\infty$$
for all $x\in\RR^n$. Thus,
$$\{{\bf W}_{1,p}\sigma=\infty\}= \left\{x\in\RR^n: {\bf W}^1_{1,p}\sigma:=\int_{0}^{1} \left(\frac{\sigma(B(x,\rho))}{\rho^{n-p}} \right)^{\frac{1}{p-1}}\frac{d\rho}{\rho}=\infty\right \}.$$ 

This yields
\begin{align*}
{\rm cap}_p(\{{\bf W}_{1,p}\sigma=\infty\}) &= \lim_{j\rightarrow\infty} {\rm cap}_p(\{{\bf W}_{1,p}\sigma=\infty\}\cap B_j(0)) \\
&= \lim_{j\rightarrow\infty} {\rm cap}_p(\{x\in B_j(0): {\bf W}^1_{1,p}(\sigma|_{B_{j+1}(0)})=\infty \}) \\
&=0.
\end{align*}
 Here we used the fact that, if $\mu\in \MM_b(\RR^n)$, 
then  ${\rm cap}_p(\{{\bf W}_{1,p}\mu=\infty \})=0$ (see \cite[Proposition 6.3.12]{AH}). 
It follows that  $\sigma(\{{\bf W}_{1,p}\sigma=\infty\})=0$, since  $\sigma<<{\rm cap}_p$.

	Let $\sigma_k$ ($k=1, 2, \dots$) be the restriction of $\sigma$ to the set $B_k(0)\cap\{ {\bf W}_{1,p}\sigma < k\}$. We then have that $\sigma_k$ weakly converges to $\sigma$, and
	$$\int_{\RR^n} {\bf W}_{1,p}\sigma_k \, d \sigma_k \leq  k\sigma(B_k(0))<+\infty.$$

Hence, 	$\sigma_k\in W^{-1, p'}(B_k(0))$ ($1/p+1/p'=1$), and for each $k>0$, there exists a unique nonnegative solution $u_k\in W^{1,p}_0(B_k(0))$ to the problem
\begin{eqnarray}\label{zzz}
\left\{\begin{array}{rcl}
-{\rm div}\, \mathcal{A}(x,\nabla u_k) &=& \sigma_k \quad {\rm in}~ B_k(0),\\
u_k&=&0\quad ~~~~~~~~~~~~ {\rm on~}\partial B_k(0).
\end{array}
\right.
\end{eqnarray}	

If we set $u_k=0$ in $\RR^n\setminus B_k(0)$, then  the sequence $\{u_k\}$ is non-decreasing, and by \cite[Theorem 2.1{\tiny }]{PV1},
$$u_k\leq K\, {\bf W}_{1,p}\sigma <\infty \quad d \sigma-{\rm a.e.}$$ 
By \cite[Theorem 1.17]{KM1}, it follows that the function $u\defeq\lim_{k\to \infty} u_k $ is 
$\mathcal{A}$-superharmonic  in $\RR^n$.
Moreover, $u \le K\, {\bf W}_{1,p}\sigma$, and 
consequently 
$$\liminf_{|x| \to \infty} u(x) \le K \, \liminf_{|x| \to \infty} {\bf W}_{1,p}\sigma(x)=0.$$
  Thus, $u$ is an $\mathcal{A}$-superharmonic solution of \eqref{Basic-PDE}. 

To show the minimality of $u$, let $v$ be another $\mathcal{A}$-superharmonic solution of \eqref{Basic-PDE}. From the construction of $u$, it is enough to show that 
$u_k\leq v$ for any $k\geq 1$. To  this end, let $\nu_j$, $j=1,2, \dots$, be the Riesz measure of ${\rm min}\{v,j\}$. Since $v$ is $\mathcal{A}$-superharmonic, it is also a local renormalized solution
to $-{\rm div}\, \mathcal{A}(x, \nabla v)    = \sigma$ in $\RR^n$   (see \cite{KKT}). Hence, by a result of \cite{DMOP, BV} and the fact that $\sigma<<{\rm cap}_p$, we obtain 
\begin{equation*}
\nu_j=\sigma|_{\{v<j\}} + \alpha_j 
\end{equation*}
for $\alpha_j\in \MM(\RR^n)$ concentrated in the set $\{v=j\}$. 

Using the estimate $v\leq K{\bf W}_{1,p}\sigma$, we deduce   
$$\nu_j \geq  \sigma|_{\{v<j\}} \geq \sigma|_{\{{ K \bf W}_{1,p}\sigma <j\}}\geq \sigma|_{\{{\bf W}_{1,p}\sigma <k\}}\geq \sigma_k,$$
provided $ j/K> k$. Since $u_k\in W^{1,p}_0(B_k(0))$ and ${\rm min}\{v,j\}\in W^{1,p}(B_k(0))$, by the comparison principle (see \cite[Lemma 5.1]{CV2}), we estimate 
$$u_k\leq {\rm min}\{v,j\}\leq v,$$
provided $j\geq K k$. Thus, $u=\lim_{k \to \infty} u_k \le v$. 
This completes the proof of the theorem.
\end{proof}

 The proof of the minimality of $u$ above can be  modified to obtain the following comparison principle.
\begin{theorem}[Comparison Principle]\label{comprin}	
	Let  $\sigma, \tilde{\sigma} \in \MM(\RR^n)$, where $\sigma \leq \tilde{\sigma}$ and $\sigma<< {\rm cap}_p$,  $1<p<n$. Then $u \le \tilde{u}$, 
where $u$ is the minimal $\mathcal{A}$-superharmonic solution  of \eqref{Basic-PDE} and $\tilde{u}$ is  any  	$\mathcal{A}$-superharmonic solution  of \eqref{Basic-PDE} with datum $\tilde{\sigma}$
in place of $\sigma$.
\end{theorem}
\begin{proof}
	Let $\sigma'_k$, $k=1,2, \dots$, be the restriction of $\sigma$ to the set $B_k(0)\cap\{ {\bf W}_{1,p}\tilde{\sigma} < k\}$. Since $\sigma<<{\rm cap}_p$ we  have that $\sigma'_k$ weakly converges to $\sigma$.  Moreover, as ${\bf W}_{1,p}\sigma'_k\leq {\bf W}_{1,p}\sigma\leq  {\bf W}_{1,p}\tilde{\sigma} < k$ on the set $\{ {\bf W}_{1,p}\tilde{\sigma} < k\}$, it follows that 
	$$\int_{\RR^n} {\bf W}_{1,p}\sigma'_k  \, d \sigma'_k\leq  k \, \sigma(B_k(0))<+\infty.$$
	
	Hence,	$\sigma'_k\in W^{-1, \frac{p}{p-1}}(B_k(0))$, and for each $k>0$ there exists a unique nonnegative solution $u'_k\in W^{1,p}_0(B_k(0))$ to the problem
	\begin{eqnarray*}
	\left\{\begin{array}{rcl}
	-{\rm div}\, \mathcal{A}(x,\nabla u'_k) &=& \sigma'_k \quad {\rm in}~ B_k(0),\\
	u'_k&=&0\quad ~~~~~~~~~~~~ {\rm on~}\partial B_k(0).
	\end{array}
	\right.
	\end{eqnarray*}	
	
	Letting $u'_k=0$ in $\RR^n\setminus B_k(0)$, we have  that the sequence $\{u'_k\}$ is non-decreasing, and by \cite[Theorem 2.1]{PV1},
	$$u'_k\leq K\, {\bf W}_{1,p}\sigma.$$ 
	
	Then  $u'_k$ converges pointwise to an 
	$\mathcal{A}$-superharmonic solution $u'$ of \eqref{Basic-PDE} 
	by \cite[Theorem 1.17]{KM1}. 
	On the other hand, by the comparison principle of \cite[Lemma 5.1]{CV2}, we have 
	$$u'_k\leq u_k, \qquad \forall k\geq 1,$$
	where $u_k$ is defined  in \eqref{zzz}. Hence,  letting $k\rightarrow\infty$, we get $u'\leq u$, which yields  $u'=u$ by the minimality of $u$.

	 We now let $\tilde{\sigma}_j$ ($j=1,2, \dots$) be the Riesz measure of ${\rm min}\{\tilde{u},j\}$. 
	Recall that the Riesz measure $\tilde{\sigma}$ of $\tilde{u}$ can be decomposed as
	$$\tilde{\sigma}=\tilde{\sigma}_0 +\tilde{\sigma}_s,$$ 
	where $\tilde{\sigma}_0\in \MM(\RR^n)$, 
	$\tilde{\sigma}_0<<{\rm cap}_p $, and $\tilde{\sigma}_s\in \MM(\RR^n)$ is concentrated on a set of zero $p$-capacity.
	Then by a result of \cite{DMOP, BV}, we have   
	\begin{equation}\label{truncatemeasure-0}
	\tilde{\sigma}_j=\tilde{\sigma}_0|_{\{\tilde{u}<j\}} + \tilde{\alpha}_j, 
	\end{equation}
	where $\tilde{\alpha}_j \in \MM(\RR^n)$ is concentrated in the set $\{\tilde{u}=j\}$. On the other hand, since 
	$\tilde{\sigma}_s(\{\tilde{u}<\infty\}) =0$ (see \cite[Lemma 2.9]{KKT}), we can rewrite \eqref{truncatemeasure-0} as 
	\begin{equation}\label{truncatemeasure}
	\tilde{\sigma}_j=\tilde{\sigma}|_{\{\tilde{u}<j\}} + \tilde{\alpha}_j. 
	\end{equation}
	
	Now using the estimate $\tilde{u}\leq K{\bf W}_{1,p}\tilde{\sigma}$ and \eqref{truncatemeasure}, we have  
	$$\tilde{\sigma}_j \geq  \tilde{\sigma}|_{\{\tilde{u}<j\}} \geq \sigma|_{\{{ K \bf W}_{1,p}\tilde{\sigma} <j\}}\geq \sigma|_{\{{\bf W}_{1,p}\tilde{\sigma} <k\}}\geq \sigma'_k,$$
	provided $ j/K> k$. 
	
	Since $u'_k\in W^{1,p}_0(B_k(0))$ and ${\rm min}\{\tilde{u},j\}\in W^{1,p}(B_k(0))$, by the comparison principle of \cite[Lemma 5.1]{CV2} we find 
	$$u'_k\leq {\rm min}\{\tilde{u},j\}\leq \tilde{u},$$
	provided we choose a $j$ such that $j\geq K k$. 
	Letting  $k\rightarrow\infty$, we obtain $u\leq \tilde{u}$ as desired.
\end{proof}

Theorem \ref{miniexist} justifies the existence (and hence uniqueness) of the minimal $\mathcal{A}$-superharmonic solution to \eqref{Basic-PDE} provided condition  \eqref{finiteness}  holds and $\sigma << {\rm cap}_p$. 
It is not known if condition \eqref{finiteness} alone is enough for the existence of the minimal solution. It is also not known if under condition \eqref{finiteness}
 and $\sigma << {\rm cap}_p$ all  $\mathcal{A}$-superharmonic solutions to \eqref{Basic-PDE} coincide with the minimal solution. For a partial result in this direction, see Theorem \ref{us} below.
  
We now introduce a new notion of a solution so that  uniqueness is guaranteed for all nonnegative locally finite measures $\sigma$ such that $\sigma<<{\rm cap}_p$.   
Our definition is an adaptation  of the notion of the reachable solution of \cite[Definition 2.3]{DMM}.

\begin{definition}\label{reach}
	Let  $\sigma \in \MM(\RR^n)$.
	We say that a function $u: \RR^n \rightarrow [0, +\infty]$  is an  $\mathcal{A}$-superharmonic reachable solution to equation \eqref{Basic-PDE} 
	if $u$ is an	$\mathcal{A}$-superharmonic solution of \eqref{Basic-PDE}, and there exist two sequences $\{u_i\}$ and $\{\sigma_i\}$, $i=1,2, \dots$,  such that 
	
	\noindent ${\rm (i)}$ Each $\sigma_i\in \MM(\RR^n)$ is  compactly  supported  in $\RR^n$, and  $\sigma_i \leq  \sigma $; 
	
	\noindent ${\rm (ii)}$ Each $u_i$ is 	an $\mathcal{A}$-superharmonic solution of \eqref{Basic-PDE} with datum $\sigma_i$ in place of $\sigma$;

	\noindent ${\rm (iii)}$ $u_i\rightarrow u$ a.e. in $\RR^n$.
\end{definition}

\begin{remark} The notion of reachable solution was introduced in \cite{DMM} for equations over bounded domains with finite measure data. It is also related to the notion of SOLA (Solution Obtained as Limit of Approximations) of 
	\cite{Da} for $L^1$ data over bounded domains.		By  ${\rm (iii)}$ and the weak continuity result of \cite{TW}, we see that   $\sigma_i \rightarrow \sigma$ weakly as measures in $\RR^n$.
	The extra  requirement $\sigma_i \leq \sigma$ in our definition plays an important role in the proof of uniqueness  in the case when the datum $\sigma$ is absolutely continuous with respect to ${\rm cap}_p$. 
\end{remark}

\begin{theorem}\label{EandU}
	Suppose $\sigma\in \MM(\RR^n)$, and suppose 
	\eqref{finiteness} holds. Then there exists an  $\mathcal{A}$-superharmonic reachable solution to \eqref{Basic-PDE}. Moreover, if  additionally $\sigma<< {\rm cap}_p$, then any 
	$\mathcal{A}$-superharmonic reachable solution is unique and coincides with the minimal solution.  
\end{theorem}

\begin{proof} {\bf Existence:} Suppose that \eqref{finiteness} holds. Then  ${\bf W}_{1,p}\sigma <+\infty$ quasi-everywhere and hence almost everywhere.
	For each $i=1,2, \dots$, let $u^{j}_{i}$ be an $\mathcal{A}$-superharmonic  renormalized solution (see \cite{DMOP}) to 
	\begin{eqnarray*}
	\left\{\begin{array}{rcl}
	-{\rm div}\, \mathcal{A}(x, \nabla u_i^j) &=& \sigma|_{B_{i}(0)} \quad {\rm in}~ B_j(0),\\
	u_i^j&=&0\quad ~~~~~~~~~~~~ {\rm on~}\partial B_j(0).
	\end{array}
	\right.
	\end{eqnarray*}

	Note that $\sigma|_{B_{i}(0)}\leq \sigma$ and $\sigma|_{B_{i}(0)}\rightarrow \sigma$ weakly as measures in $\RR^n$. Also,
	by \cite{PV1}, we have 	
	$$u_i^j\leq K\, {\bf W}_{1,p}(\sigma|_{B_{i}(0)}). $$ 
	
	Hence, by \cite[Theorem 1.17]{KM1}, there exist an $\mathcal{A}$-superharmonic function $u_i$ in $\RR^n$ with 
	\begin{equation}\label{uibound}
	u_i\leq K\, {\bf W}_{1,p}(\sigma|_{B_i(0)})\leq K\, {\bf W}_{1,p}\sigma <+\infty {\rm ~a.e.},
	\end{equation}
	and a subsequence $\{u_i^{j_k}\}_{k}$
	such that $u_{i}^{j_k}\rightarrow u$ and $D u_{i}^{j_k} \rightarrow D u_{i}$ a.e. as $k\rightarrow \infty$.  
	These estimates yield that the Riesz measure of $u_i$ is $\sigma|_{B_{i}(0)}$ and $$\liminf_{|x|\rightarrow \infty}\,  u_i = 0.$$
	
	Using again \cite[Theorem 1.17]{KM1} and \eqref{uibound}, we find a subsequence of $\{u_i\}$ that converges a.e. to 
	an $\mathcal{A}$-superharmonic reachable solution $u$ of \eqref{Basic-PDE}. 
	
	\noindent {\bf Uniqueness:} We now assume further that $\sigma<< {\rm cap}_p$. Let $u$ be an $\mathcal{A}$-superharmonic reachable solution in the sense of Definition \ref{reach} with approximating sequences $\{{u_i}\}$ and $\{\sigma_i\}$.  Let us fix an $i\in \{1,2,\dots\}$. Then there exists a positive integer $N=N(i)$ such that ${\rm supp}(\sigma_i)\subset B_N(0)$. Let   $v$  be the minimal $\mathcal{A}$-superharmonic solution to \eqref{Basic-PDE}. Also, let   $v_N$  be the minimal $\mathcal{A}$-superharmonic solution to \eqref{Basic-PDE} with datum  $\sigma|_{B_N(0)}$ in place of $\sigma$. We have, by Theorem \ref{comprin}, 
	$$u\geq v \geq v_N.$$
	
	Thus, as $u_j \rightarrow u$ a.e., it is enough to show that 
	\begin{equation}\label{uileqvn}
	v_N\geq u_i.
	\end{equation}
	
	Note that  since $\sigma_i\leq \sigma$ and ${\rm supp}(\sigma_i)\subset B_N(0)$ we have that 
	\begin{equation}\label{sigmaivssigmaB}
	\sigma_i\leq \sigma|_{B_N(0)}.
	\end{equation}

	For $R>0$, let $0\leq \Theta=\Theta_R\leq 1$ be a  cutoff function such that 
	$$\Theta\in C_0^\infty(B_R(0)), \quad  \quad \Theta \equiv 1 {\rm ~ on~} B_{R/2}(0), \quad {\rm and} \quad
	 |\nabla \Theta|\leq C/R.$$
	For any $k>0$, we set 
	\begin{equation*}
	T^{+}_k(t) =\left\{ \begin{array}{ll}
	t  \quad \text{if } 0\leq t\leq k, \\
	k \quad \text{if } t>k,\\
	0 \quad \text{if } t<0.
	\end{array}
	\right.
	\end{equation*}
Also, for any $m>0$, we define the following Lipschitz function with compact support on $\RR$:
\begin{equation*}
h_m(t) =\left\{ \begin{array}{ll}
1  \quad \text{if } 0\leq |t|\leq m, \\
0 \quad \text{if } |t|\geq 2m,\\
-\frac{t}{m} + 2  \quad \text{if } m<t<2m,\\
\frac{t}{m} + 2  \quad \text{if } -2m<t<-m.
\end{array}
\right.
\end{equation*}

As $u_i$ and $v_N$ are both local renormalized solutions (see \cite{BV, KKT}), we may use 
$$h_m(u_i) h_m(v_n)T^{+}_k(u_i-v_N)\Theta, \qquad m, k>0,$$
as test functions and thus obtaining
\begin{align*}
&\int_{\RR^n} \mathcal{A}(x, D u_i) \cdot \nabla \left[h_m(u_i) h_m(v_N)T^{+}_k(u_i-v_N)\Theta\right]dx \\
&\qquad = \int_{\RR^n} h_m(u_i) h_m(v_n)T^{+}_k(u_i-v_N)\Theta d\sigma_i,
\end{align*}
and 
\begin{align*}
&\int_{\RR^n} \mathcal{A}(x, D v_N)  \cdot \nabla \left [h_m(u_i) h_m(v_N)T^{+}_k(u_i-v_N)\Theta\right] dx\\
&\qquad = \int_{\RR^n} h_m(u_i) h_m(v_n)T^{+}_k(u_i-v_N)\Theta d\sigma|_{B_N(0)}.
\end{align*}

	Let
	$$I= \int_{\RR^n} \mathcal{A}(x, D u_i) \cdot \nabla\left[h_m(u_i) h_m(v_N)T^{+}_k(u_i-v_N)\Theta\right]dx,$$
	and 
	$$II=\int_{\RR^n} \mathcal{A}(x, D v_N) \cdot \nabla\left[h_m(u_i) h_m(v_N)T^{+}_k(u_i-v_N)\Theta\right]dx.$$
	
	Then by \eqref{sigmaivssigmaB} we have 
	\begin{equation}\label{I-II}
	 I -II\leq 0.
	\end{equation}
	
	On the other hand, we can write 
	\begin{align*}
	&I -II\\ 
	&= \int_{\RR^n} \left[\mathcal{A}(x, D u_i)- \mathcal{A}(x, D v_N)\right] \cdot \nabla T^{+}_k(u_i-v_N)  \, h_m(u_i) \, h_m(v_N) \, \Theta \, dx\\
	& \quad + \int_{\RR^n} \left[\mathcal{A}(x, D u_i)- \mathcal{A}(x, D v_N)\right]   \cdot D u_i \, h_m'(u_i) \, T^{+}_k(u_i-v_N)  \, h_m(v_N) \, \Theta \, dx \nonumber\\
	& \quad + \int_{\RR^n} \left[\mathcal{A}(x, D u_i)- \mathcal{A}(x, D v_N)\right] \cdot D v_N \, h_m'(v_N) \, T^{+}_k(u_i-v_N) \, h_m(u_i) \, \Theta \, dx \nonumber\\
	&\quad + \int_{\RR^n} \left[\mathcal{A}(x, D u_i)- \mathcal{A}(x, D v_N)\right] \cdot \nabla \Theta \, T^{+}_k(u_i-v_N) \, h_m(u_i) \, h_m(v_N) \, dx. \nonumber
	\end{align*}
	
	Thus, in view of \eqref{I-II}, it follows that 
	\begin{align*}
	 &\int_{\{ 0<u_i -v_N<k\}} \left[\mathcal{A}(x, D u_i)- \mathcal{A}(x, D v_N)\right]  \cdot (D u_i- D v_N)  \, h_m(u_i) \, h_m(v_N) \, \Theta \, dx\\
	& \leq - \int_{\RR^n} \left[\mathcal{A}(x, D u_i)- \mathcal{A}(x, D v_N)\right] \cdot D u_i \, h_m'(u_i) \, T^{+}_k(u_i-v_N)  \, h_m(v_N) \, \Theta \, dx \nonumber\\
	&\quad - \int_{\RR^n} \left[\mathcal{A}(x, D u_i)- \mathcal{A}(x, D v_N)\right] \cdot D v_N \, h_m'(v_N) \, T^{+}_k(u_i-v_N) \, h_m(u_i) \, \Theta \, dx \nonumber\\
	&\quad - \int_{\RR^n} \left[\mathcal{A}(x, D u_i)- \mathcal{A}(x, D v_N)\right] \cdot \nabla \Theta \, T^{+}_k(u_i-v_N) \, h_m(u_i) \, h_m(v_N) \, dx \nonumber\\
	&\quad =: A_m+B_m+C_m.
	\end{align*}
	
	To estimate $|A_m|$, we observe  that $|h_m(t)|\leq 1$ and $|h'_m(t)|\leq 1/m$. Hence, 
\begin{align*}
|A_m| &\leq \beta \frac{k}{m} \int_{\{ m<u_i<2m, \, 0<v_N<2m\}} \left[|D u_i|^{p-1}+ |D v_N|^{p-1}\right]  |D u_i| \, \Theta \, dx \nonumber\\
&\leq C \frac{k}{m} \int_{\{ 0<u_i<2m, \, 0<v_N<2m\}} \left[|D u_i|^{p}+ |D v_N|^{p}\right] \, \Theta \, dx. \nonumber
\end{align*}

On the other hand, using $T^{+}_{2m}(u_i) \, \Theta$ as a test function for the  equation of $u_i$ and invoking condition \eqref{structure}, we estimate  
\begin{align*}
\alpha \int_{0<u_i<2m} |D u_i|^p \, \Theta \, dx  &\leq  \int_{\RR^n} T^{+}_{2m}(u_i) \, \Theta \, d\sigma_i 
\\ & +\beta \int_{\RR^n} |D u_i|^{p-1} T^{+}_{2m}(u_i) \, |\nabla \Theta| \, dx. \nonumber
\end{align*}

Since 	$T^{+}_{2m}(u_i)/m\leq 2$, and $T^{+}_{2m}(u_i)/m$ converges to zero quasi-everywhere, we deduce 
$$ \lim_{m\rightarrow \infty}  \frac{1}{m}\int_{0<u_i<2m} |D u_i|^p \, \Theta \, dx =0.$$

Similarly, 
$$ \lim_{m\rightarrow \infty}  \frac{1}{m}\int_{0<v_N<2m} |D v_N|^p \,  \Theta \, dx =0.$$
Hence,
\begin{align}\label{A}
\lim_{m\rightarrow \infty} |A_m| =0. 
\end{align}

A similar argument gives  
\begin{align*}
\lim_{m\rightarrow \infty} |B_m| =0. 
\end{align*}

To estimate $|C_m|$, we first use the pointwise bound \eqref{Wolffbound} to obtain
\begin{align*} 	
|C_m| \leq  \frac{c}{R} \int_{A_R} \left[ |D u_i|^{p-1} +  |D v_N|^{p-1} \right] \min\{{\bf W}_{1,p}(\sigma|_{B_N(0)}) ,k\} \, dx,
\end{align*}
where $A_R$ is the annulus  $$A_R=\{R/2 < |x| <R\}.$$ 	

Note that for $R>4N$ we have 
$${\bf W}_{1,p}(\sigma|_{B_N(0)})(x) = \int_{R/4}^\infty \left[\frac{\sigma(B_t(x)\cap B_N(0))}{t^{n-p}}\right]^{\frac{1}{p-1}} \frac{dt}{t}\approx R^{\frac{p-n}{p-1}}$$
for all $x\in A_R$. Thus, 
\begin{align*} 	
|C_m| &\leq  c\, R^{\frac{p-n}{p-1}} R^{-1} \int_{A_R} \left[ |D u_i|^{p-1} +  |D v_N|^{p-1} \right] dx\\
&\leq c\,  R^{\frac{p-n}{p-1}} R^{-1} R^{n-p+1} \left[ (\inf_{A_R} u_i)^{p-1} + (\inf_{A_R} v_N)^{p-1}  \right]\\
&\leq c\,  R^{\frac{p-n}{p-1}} R^{-1} R^{n-p+1} R^{p-n} = c\, R^{\frac{p-n}{p-1}}, 
\end{align*}
where we used the Caccioppoli inequality  and the weak Harnack inequality in the second bound. This gives
\begin{align}\label{C}
\lim_{R\rightarrow\infty}\limsup_{m\rightarrow \infty} |C_m| =0. 
\end{align}

Since $h_m(u_i) \, h_m(v_N) \rightarrow 1$ a.e. as $m\rightarrow\infty$, and $\Theta(x)\rightarrow 1$ everywhere as $R\rightarrow\infty$, it follows from \eqref{A}--\eqref{C} and Fatou's lemma that 
\begin{align*}
\int_{\{ 0<u_i -v_N<k\}} \left[\mathcal{A}(x, D u_i)- \mathcal{A}(x, D v_N)\right]  \cdot (D u_i- D v_N) \, dx \leq 0.
\end{align*}

Letting $k\rightarrow\infty$, we deduce 
\begin{align*}
\int_{\{ u_i -v_N>0\}} \left[\mathcal{A}(x, D u_i)- \mathcal{A}(x, D v_N)\right]  \cdot (D u_i- D v_N) dx \leq 0.
\end{align*}

Since the integrand is strictly positive whenever $D u_i \not=D v_N$,   we infer that $D u_i =D v_N$ a.e. on
the set $\{u_i-v_N>0\}$.

We next claim that the function $T_k^{+}(u_i-v_N)$ belongs to $W^{1,p}_{\rm loc}(\RR^n)$ for any $k>0$. To see this, for any $m>k$, we compute
\begin{align}\label{graddetail}
&\nabla T^{+}_k(T^{+}_m(u_i)-T^{+}_m(v_N))\\
&=\left[\nabla T^{+}_m(u_i)-\nabla T^{+}_m(v_N)\right] \chi_{\{0<T^{+}_m(u_i)-T^{+}_m(v_N)<k\}}\nonumber\\
&=\left[Du_i\chi_{\{0<u_i<m\}} - Dv_N\chi_{\{0<v_N<m\}}\right] \chi_{\{0<T^{+}_m(u_i)-T^{+}_m(v_N)<k\}}\nonumber\\
&=\left[Du_i\chi_{\{0<u_i<m\}} - Dv_N\chi_{\{0<v_N<m\}}\right] \chi_{\{0<m-v_N<k, u_i\geq m, v_N<m\}}\nonumber\\
&= - Dv_N\chi_{\{0<v_N<m\}} \chi_{\{0<m- v_N<k, u\geq m, v<m\}},\nonumber
\end{align}
 where $\chi_A$ is the characteristic function of a set $A$. Thus,
\begin{align*}
\int_{\RR^n} |\nabla T^{+}_k(T^{+}_m(u_i)-T^{+}_m(v_N))|^p \, \Theta \, dx\leq  \int_{\{ m-k<v_N<m\}} |D v_N|^p \, \Theta \, dx.
\end{align*}

On the other hand, using $H_{m,k}(v_N)\Theta$ as a test function for the equation of $v_N$, where 
\begin{equation*}
H_{m,k}(t) =\left\{ \begin{array}{ll}
1  \quad \text{if } 0\leq |t|\leq m-k, \\
0 \quad \text{if } |t|\geq m,\\
-\frac{t}{k} + \frac{m}{k}  \quad \text{if } m-k<t<m,\\
\frac{t}{k} +  \frac{m}{k} \quad \text{if } -m<t<-(m-k),
\end{array}
\right.
\end{equation*}
we have  
\begin{align*}
\frac{\alpha}{k}\int_{\{m-k<v_N<m\}} |D v_N|^p \, \Theta \, dx  &\leq  \int_{\RR^n} \Theta \, d\sigma|_{B_N(0)} +\beta \int_{\RR^n} |D v_N|^{p-1} \,  |\nabla \Theta| \, dx. \nonumber
\end{align*}

Thus, for each fixed $k>0$, the sequence $\{T^{+}_k(T^{+}_m(u_i)-T^{+}_m(v_N))\}_{m}$ is uniformly bounded in $W^{1,p}_{\rm loc}(\RR^n)$. Since 
$T^{+}_k(T^{+}_m(u_i)-T^{+}_m(v_N))\rightarrow T_k^{+}(u_i-v_N)$ a.e. as $m\rightarrow\infty$, we see that $T_k^{+}(u_i-v_N)\in W^{1,p}_{\rm loc}(\RR^n)$.
 
 We are now ready to complete the proof of the theorem. Since $T_k^{+}(u_i-v_N)=T^{+}_k(T^{+}_m(u_i)-T^{+}_m(v_N))$
 a.e. on the set $\{ u_i<m, v_N<m\}$ and the two functions belong to $W^{1,p}_{\rm loc}(\RR^n)$, by \eqref{graddetail} we have
 \begin{align*}
 \nabla T_k^{+}(u_i-v_N)=\nabla T^{+}_k(T^{+}_m(u_i)-T^{+}_m(v_N))=0
 \end{align*}
 a.e. on the set $\{ u_i<m, v_N<m\}$ for any $m>0$. Thus, $\nabla T_k^{+}(u_i-v_N)=0$ a.e. in $\RR^n$, which implies the existence of a constant $\kappa\geq 0$
 such that 
$$\max \{u_i-v_N, 0\}=\kappa$$
a.e. in the entire space $\RR^n$.
 Note that if $\kappa\not=0$, then $u_i=v_N+\kappa$ in $\RR^n$, which violates the condition at infinity, $\liminf_{|x|\rightarrow\infty} u_i(x)=0$. It follows that $\kappa=0$, which yields \eqref{uileqvn}, as desired. 
\end{proof}

The following version of the comparison principle in $\RR^n$ 
is an immediate consequence of Theorems \ref{comprin} and \ref{EandU}. 

\begin{corollary}\label{cor-1}	
	Let  $\sigma, \tilde{\sigma} \in \MM(\RR^n)$, where $\sigma \leq \tilde{\sigma}$ and $\sigma<< {\rm cap}_p$,  $1<p<n$.  
	Let  $u$  be an $\mathcal{A}$-superharmonic reachable solution of \eqref{Basic-PDE}, and $\tilde{u}$ 
	 any  	$\mathcal{A}$-superharmonic solution of \eqref{Basic-PDE} with datum $\tilde{\sigma}$
	in place of $\sigma$. Then $u \le \tilde{u}$ in $\RR^n$. 	
\end{corollary}

For $\sigma\in \MM(\RR^n)$ such that that $\sigma<<{\rm cap}_p$, sometimes it is desirable to know when  an $\mathcal{A}$-superharmonic solution to \eqref{Basic-PDE} is also the $\mathcal{A}$-superharmonic reachable  solution to \eqref{Basic-PDE}, and hence also the minimal 
$\mathcal{A}$-superharmonic solution to \eqref{Basic-PDE}. The following theorem provides some sufficient conditions in terms of the weak integrability of the gradient of the solution, or in terms of the finiteness of the datum $\sigma$.

\begin{theorem} \label{us}
	Let  $\sigma\in \MM(\RR^n)$,  where $\sigma<< {\rm cap}_p$. Suppose that 
any one of the following conditions holds:

\noindent {\rm (i)}	$|D u| \in L^{\gamma,\infty}(\RR^n)$ for some $(p-1)n/(n-1)\leq  \gamma < p$, where $L^{\gamma,\infty}(\RR^n)$ is the weak $L^\gamma$ space in $\RR^n$;

\noindent {\rm (ii)} 	$|D u| \in L^{p}(\RR^n)$;
	
\noindent {\rm (iii)} $\sigma \in \MM_b(\RR^n)$. 

	Then any  $\mathcal{A}$-superharmonic solution $u$ to the equation \eqref{Basic-PDE}
coincides with the minimal  $\mathcal{A}$-superharmonic solution.
\end{theorem}
\begin{proof}
Let $v$ be the minimal  $\mathcal{A}$-superharmonic solution of \eqref{Basic-PDE}. Our goal is to show that $u\leq v$ a.e.
Let $\Theta(x)=\Theta_R(x), R>0$, $T^{+}_k(t), k>0$, and $h_m(t), m>0$ be as in the proof of Theorem \ref{EandU}.
Then arguing as in the proof of Theorem \ref{EandU}, with $u$ in place of $u_i$ and $v$ in place of $v_N$, we have 

\begin{align*}
&\int_{\{ 0<u_i -v_N<k\}} \left[\mathcal{A}(x, D u)- \mathcal{A}(x, D v)\right]  \cdot (D u- D v) \,  h_m(u) \, h_m(v)\, \Theta \, dx\\
&\leq  A_m+B_m+C_m,
\end{align*}
where now 
$$A_m=- \int_{\RR^n} \left[\mathcal{A}(x, D u)- \mathcal{A}(x, D v)\right] \cdot  D u \, h_m'(u) \, T^{+}_k(u-v)  \, h_m(v) \, \Theta \, dx,$$
$$B_m=- \int_{\RR^n} \left[\mathcal{A}(x, D u)- \mathcal{A}(x, D v)\right]  \cdot D v \, h_m'(v) \, T^{+}_k(u-v) \, h_m(u) \, \Theta \, dx,$$
and 
$$C_m=- \int_{\RR^n} \left[\mathcal{A}(x, D u)- \mathcal{A}(x, D v)\right] \cdot \nabla \Theta \, T^{+}_k(u-v) \,  h_m(u) \, h_m(v) \, dx.$$

As in the proof of Theorem \ref{EandU}, we have
\begin{equation}\label{ABzero}
\lim_{m\rightarrow\infty}(|A_m|+ |B_m|)=0.
\end{equation} 

As for $C_m$, we have 
\begin{align*} 	
|C_m| \leq  \frac{c}{R} \int_{A_R} \left[ |D u|^{p-1} +  |D v|^{p-1} \right] \, \min\{{\bf W}_{1,p}\sigma ,k\} \, dx,
\end{align*}
where, as above, $A_R$ is the annulus  $$A_R=\{R/2 < |x| <R\}.$$

Suppose now that condition (i) holds. Then $|D u|^{p-1} \in L^{\frac{q}{q-1}, \infty}(\RR^n)$ for some $q\in (p,n]$. Set
$$m=\frac{n(p-1)q}{n(q-1)-(p-1)q}>0,$$
and note that $
{\bf W}_{1,p}\sigma \in L^{m, \infty}(\RR^n). 
$
A proof of this fact in the `sublinear' case  $(p-1)q/(q-1)\leq 1$ can be found in \cite{Ph17}. 

We have that either $m\leq  q$ or $m>q$. In the case $m\leq q$, for any $\epsilon>0$ we find
\begin{align*}
&\frac{1}{R}\int_{A_R} |D u|^{p-1}  \min\{ {\bf W}_{1,p}\sigma, k\} dx  \\
&\qquad \leq \frac{1}{R} \norm{|D u|^{p-1}}_{L^{\frac{q}{q-1},\infty}(A_R)} \norm{\min\{ {\bf W}_{1,p}\sigma, k\}}_{L^{q,1}(A_R)}\\
&\qquad \leq \frac{1}{R} \norm{|D u|^{p-1}}_{L^{\frac{q}{q-1},\infty}(\RR^n)} \norm{\min\{ {\bf W}_{1,p}\sigma, k\}}_{L^{q+\epsilon, \infty}(A_R)} |A_R|^{\frac{\epsilon}{q(q+\epsilon)}} \\
&\qquad \leq C k^{\frac{q+\epsilon -m}{q+\epsilon}}  \norm{|D u|^{p-1}}_{L^{\frac{q}{q-1},\infty}(\RR^n)} \norm{ {\bf W}_{1,p}\sigma}_{L^{m, \infty}(\RR^n)}^{\frac{m}{q+\epsilon}}  R^{\frac{ n \epsilon}{q(q+\epsilon)}-1}.
\end{align*}
Here we shall choose $\epsilon>0$ such that 
$$\frac{ n \epsilon}{q(q+\epsilon)}-1<0.$$

In the case $m> q$, we have

\begin{align}\label{ca2}
&\frac{1}{R}\int_{A_R} |D u|^{p-1}  \min\{ {\bf W}_{1,p}\sigma, k\} dx  \\
&\qquad \leq \frac{1}{R} \norm{|D u|^{p-1}}_{L^{\frac{q}{q-1},\infty}(A_R)} \norm{\min\{ {\bf W}_{1,p}\sigma, k\}}_{L^{q,1}(A_R)}\nonumber\\
&\qquad \leq \frac{1}{R} \norm{|D u|^{p-1}}_{L^{\frac{q}{q-1},\infty}(\RR^n)} \norm{ {\bf W}_{1,p}\sigma}_{L^{m, \infty}(\RR^n)} |A_R|^{\frac{m-q}{q m}}\nonumber\\
&\qquad \leq  C \norm{|D u|^{p-1}}_{L^{\frac{q}{q-1},\infty}(\RR^n)} \norm{ {\bf W}_{1,p}\sigma}_{L^{m, \infty}(\RR^n)} R^{\frac{(m-q)n}{q m}-1}.\nonumber
\end{align}

Note that, since $q < p$, 
$$\frac{(m-q)n}{mq}-1=\frac{n(p-1) - n(q-1)+ (p-1)q}{(p-1)q} -1 <0.$$

Hence, in both cases we  have, for any fixed $k>0$,
 \begin{align}\label{ulim}
 \lim_{R\rightarrow \infty}\frac{1}{R}\int_{A_R} |D u|^{p-1}  \min\{ {\bf W}_{1,p}\sigma, k\} dx  =0,
 \end{align}
 and likewise,
\begin{align}\label{vlim}
\lim_{R\rightarrow \infty}\frac{1}{R}\int_{A_R} |D v|^{p-1}  \min\{ {\bf W}_{1,p}\sigma, k\} dx  =0.
\end{align}

On the other hand,  suppose now that condition (ii) holds, i.e.,  $|D u| \in L^{p}(\RR^n)$. Then ${\bf W}_{1,p}\sigma \in L^{\frac{np}{n-p}}(\RR^n)$, and  as in \eqref{ca2} we have 
\begin{align*}
&\frac{1}{R}\int_{A_R} |D u|^{p-1}  \min\{ {\bf W}_{1,p}\sigma, k\} dx  \\
&\qquad \leq  C \norm{|D u|^{p-1}}_{L^{\frac{p}{p-1}}(A_R)} \norm{ {\bf W}_{1,p}\sigma}_{L^{\frac{np}{n-p}}(A_R)},
\end{align*}
and likewise for $v$. Thus \eqref{ulim} and \eqref{vlim} also hold under condition (ii).
 
Finally, suppose that ${\rm (iii)}$ holds. For any  $1<r<\frac{n}{n-1}$ and $\epsilon\in(0,1)$ such that $\ep \frac{r}{r-1}<\frac{n(p-1)}{n-p}$, we have 
 \begin{align*} 	
  & \frac{1}{R} \int_{A_R}  |D u|^{p-1}  \min\{{\bf W}_{1,p}\sigma ,k\} dx\\
  &\leq k^{1-\epsilon} \frac{1}{R} \left(\int_{B_R(0)} |Du|^{(p-1)r} dx \right)^{\frac{1}{r}} \left(\int_{B_R(0)} ({\bf W}_{1,p}\sigma)^{\frac{\epsilon\, r}{r-1}} dx \right)^{\frac{r-1}{r}}\\
  &\leq C k^{1-\epsilon} R^{\frac{n}{r}-1} \left(\frac{\inf_{B_R(0)} u}{R}\right)^{p-1} R^{n\frac{r-1}{r}} \left(\inf_{B_R(0)} u \right)^{\epsilon},  
 \end{align*}
 where we used the Caccioppoli inequality and the weak Harnack inequality in the last bound (see \cite[Theorem 7.46]{HKM}).  
 
 Hence,  using \cite[Lemma 3.1]{CV2} we get
  \begin{align} 	\label{decayfinitemeas}
  & \frac{1}{R} \int_{A_R}  |D u|^{p-1}  \min\{{\bf W}_{1,p}\sigma ,k\} dx\\
  &\leq C k^{1-\epsilon} R^{n-p} \left(\int_{R}^\infty \left(\frac{\sigma(B_t(0))}{t^{n-p}}\right)^{\frac{1}{p-1}} \frac{dt}{t}\right)^{p-1} \left(\inf_{B_R(0)} u \right)^{\epsilon}\nonumber\\
  & \leq C k^{1-\epsilon} \sigma(\RR^n) \left(\inf_{B_R(0)} u \right)^{\epsilon}.\nonumber
  \end{align}
 
 A similar inequality  holds for $v$ in place of $u$. Thus, we see that \eqref{ulim} and \eqref{vlim}  hold under condition (iii) as well.
 
   Now \eqref{ulim} and \eqref{vlim} yield that,  for any $k>0$, we have 
 \begin{align}\label{Czero}
 \lim_{R\rightarrow \infty} \limsup_{m\rightarrow\infty} |C_m| =0.
 \end{align}

Using \eqref{ABzero} and \eqref{Czero},  we deduce   
\begin{align*}
\int_{\{ 0<u -v<k\}} \left[\mathcal{A}(x, D u)- \mathcal{A}(x, D v)\right]  \cdot (D u- D v) dx \leq 0
\end{align*}
for any $k>0$. This implies  $Du=Dv$ a.e. on the set $\{u-v>0\}$ and, as in the proof of Theorem \ref{EandU}, in view of the condition at infinity, we deduce  $u\leq v$ a.e. as desired.
\end{proof}

%

We  now provide a criterion for reachability by requiring only the finiteness of the approximating measures $\sigma_i$.   

\begin{corollary}
	Let $u$ be an  $\mathcal{A}$-superharmonic solution of \eqref{Basic-PDE}, where $\sigma \in \MM(\RR^n)$, and $\sigma<< {\rm cap}_p$. Suppose that there exist two sequences $\{u_i\}$ and $\{\sigma_i\}$, $i=1,2, \dots$,  such that the following conditions hold:

	\noindent ${\rm (i)}$ each $\sigma_i \in \MM_b(\RR^n)$, and  $\sigma_i \leq  \sigma $; 
	
	\noindent ${\rm (ii)}$ each $u_i$ is 	an $\mathcal{A}$-superharmonic solution of \eqref{Basic-PDE} with datum $\sigma_i$ in place of $\sigma$;

	\noindent ${\rm (iii)}$ $u_i\rightarrow u$ a.e. in $\RR^n$.
	
	Then $u$ is an $\mathcal{A}$-superharmonic reachable solution of  
	\eqref{Basic-PDE}, and thus coincides with the minimal solution.
\end{corollary}

\begin{proof} By Theorem \ref{us}, each $u_i$ is a reachable solution. Thus by a diagonal process argument, we see that $u$ is also a reachable solution.   
Alternatively, this can also be proved by modifying the proof of the uniqueness part in Theorem \ref{EandU}, taking into account   estimates of the form \eqref{decayfinitemeas}. 	
\end{proof}

Theorem \ref{us} formally holds under the condition $|D u| \in L^{\gamma,\infty}(\RR^n)$ for  $0< \gamma < (p-1)n/(n-1)$ as in this case $\sigma=0$. 
The proof of this fact, especially in the case $0<\gamma\leq p-1$, requires some  results obtained recently in \cite{HP4}. 

\begin{theorem}
If $u$ is an $\mathcal{A}$-superharmonic function in $\RR^n$ such that 	 $|D u| \in L^{\gamma,\infty}(\RR^n)$ for some $0< \gamma < (p-1)n/(n-1)$, then $\sigma =0$ where $\sigma$ is the Riesz measure of $u$.
\end{theorem}	

\begin{proof}
Let $Q_r(x)$, $r>0$,  denote the open cube $Q_r(x):= x+ (-r,r)^n$ with center $x\in\RR^n$ and side-length  $2r$. Using $\Phi\in C_0^\infty (Q_r(x))$, $\Phi\geq 0$, 
$\Phi=1$ on  $Q_{r/2}(x)$, and $|\nabla \Phi|\leq C/r$, as a test function we have 
\begin{align}\label{sigmaBR}
\sigma(Q_{r/2}(x)) \leq \frac{C}{r}\int_{Q_{r}(x)} |D u|^{p-1} dy.
\end{align}

Thus  if $\gamma\in (p-1, (p-1)n/(n-1))$, for any $R>0$ we use H\"older's inequality to get 
\begin{align*}
\sigma(Q_{R/2}(0)) \leq \frac{C}{R}\norm{D u}^{p-1}_{L^{\gamma,\infty}(\RR^n)} R^{n\frac{\gamma-p+1}{\gamma}}.
\end{align*}
Note that $n\frac{\gamma-p+1}{\gamma}<1$ and thus letting $R\rightarrow\infty$ we get $\sigma=0$.

We now consider the case $0<\gamma\leq p-1$. Let $\gamma_1$ be a fixed number in $(p-1), (p-1)n/(n-1))$.
By \cite[Lemma 2.3]{HP4}, for any cube $Q_\rho(x)\subset\RR^n$, we have
\begin{align*}
\left(\fint_{Q_{\rho}(x)} |D u|^{\gamma_1} dy\right)^{\frac{1}{\gamma_1}} & \lesssim \left[\frac{\sigma(Q_{3\rho/2}(x))}{\rho^{n-1}}\right]^{\frac{1}{p-1}}\\
& \qquad + \frac{1}{\rho} \inf_{q\in\RR}\left(\fint_{Q_{3\rho/2}(x)} |u-q|^{p-1} dy\right)^{\frac{1}{p-1}}.
\end{align*}

 On the other hand, by \cite[Corollary 1.3]{HP4} we find 
\begin{align*}
\frac{1}{\rho} \inf_{q\in\RR}\left(\fint_{Q_{3\rho/2}(x)} |u-q|^{p-1} dy\right)^{\frac{1}{p-1}} & \lesssim \left[\frac{\sigma(Q_{2\rho}(x))}{\rho^{n-1}}\right]^{\frac{1}{p-1}} \\
&\qquad + \left(\fint_{Q_{2 \rho}(x)} |D u|^{p-1} dy\right)^{\frac{1}{p-1}}.
\end{align*}
Note that \cite[Corollary 1.3]{HP4} is stated for $1<p<3/2$ but the argument there also works for all $1<p\leq n$  after taking into account the comparison estimates of 
\cite{Min, DM1, HP1}. 

Hence, it follows that 
 \begin{align*}
 \left(\fint_{Q_{\rho}(x)} |D u|^{\gamma_1} dy\right)^{\frac{1}{\gamma_1}} & \lesssim \left[\frac{\sigma(Q_{2\rho}(x))}{\rho^{n-1}}\right]^{\frac{1}{p-1}} + \left(\fint_{Q_{2 \rho}(x)} |D u|^{p-1} dy\right)^{\frac{1}{p-1}}\\
 &\lesssim \left(\fint_{Q_{4 \rho}(x)} |D u|^{p-1} dy\right)^{\frac{1}{p-1}},
 \end{align*}
where we used \eqref{sigmaBR} with $r=4 \rho$ in the last inequality. This allows us to employ 
a covering/iteration argument as in \cite[Remark 6.12]{Giu} to obtain  that
\begin{align}\label{reverseH}
\left(\fint_{Q_{\rho}(x)} |D u|^{\gamma_1} dy\right)^{\frac{1}{\gamma_1}}  \lesssim \left(\fint_{Q_{4 \rho}(x)} |D u|^{\epsilon} dy\right)^{\frac{1}{\epsilon}}
\end{align}
 for any $\epsilon>0$.
 
 Thus, if $0<\gamma<p-1$, in view of \eqref{sigmaBR}, \eqref{reverseH}, and H\"older's inequality,  we get 
  \begin{align*}
  \sigma(Q_{R/2}(0)) \leq \frac{C}{R}\norm{D u}^{\frac{p-1}{\gamma}}_{L^{\gamma,\infty}(\RR^n)} R^{n- n\frac{p-1}{\gamma}}\rightarrow 0,
  \end{align*}
 as $R\rightarrow\infty$. Hence, $\sigma=0$. The case $\gamma=p-1$ is treated similarly, starting with the inequality
  \begin{align*}
 \sigma(Q_{R/2}(0)) \leq \frac{C}{R} \left(\int_{Q_R(0)} |D u|^{(p-1)(1+\epsilon)} dx\right)^{\frac{1}{1+\epsilon}}  R^{\frac{n\epsilon}{1+\epsilon}}
 \end{align*}
 for a sufficiently small $\epsilon>0$.
\end{proof}

Due to the results of \cite{DM1, KuMi, DZ, HP3} (see also \cite{DM2, HP2}), under some additional regularity conditions on the nonlinearity $\mathcal{A}(x,\xi)$, one has 
\begin{equation*}
|Du(x)| \leq C \left[{\bf I}_1 \sigma(x)\right]^{\frac{1}{p-1}} \quad {\rm a.e.~} x\in\RR^n,
\end{equation*}
provided $u$ is an $\mathcal{A}$-superharmonic solution  to the equation \eqref{Basic-PDE}. This gradient estimate holds in particular  for $\mathcal{A}(x,\xi)=|\xi|^{p-2}\xi$,  i.e.,   the $p$-Laplacian $\Delta_p$, which yields the following corollary. 

\begin{corollary}
			Let  $\sigma \in \MM(\RR^n)$.  Suppose that 
		one of the following conditions holds:
			
		\noindent {\rm (i)} $\sigma<<{\rm cap}_p$ and	${\bf I}_1 \sigma \in L^{s,\infty}(\RR^n)$ for some $n/(n-1)<  s < p/(p-1)$. This holds in particular if 
		$\sigma =f\in L^{t,\infty}(\RR^n)$ for some $1<t<np/(np-n+p)$;
		
		\noindent {\rm (ii)} 	${\bf I}_1 \sigma \in L^{p/(p-1)}(\RR^n)$, i.e., $\sigma$ is of finite energy.

		Then any  $p$-superharmonic solution $u$ to the equation 
		\begin{equation*}
		\left\{ \begin{array}{ll}
		- \Delta_p u = \sigma, \quad u\geq 0  \quad \text{in } \RR^n, \\
		\displaystyle{\liminf_{|x|\rightarrow \infty}}\,  u = 0, 
		\end{array}
		\right.
		\end{equation*}
		coincides with the minimal  $p$-superharmonic solution.
\end{corollary}

Finally, we show that if the condition at infinity, $\displaystyle{\liminf_{|x|\rightarrow \infty}}\,  u = 0$ in \eqref{Basic-PDE}, is replaced with the stronger one
$\displaystyle{\lim_{|x|\rightarrow \infty}}\,  u = 0$, then all $\mathcal{A}$-superharmonic solutions are indeed reachable. 

\begin{theorem}\label{thm-lim} Suppose that  $u$ is an $\mathcal{A}$-superharmonic solution  of  the equation
	\begin{equation}\label{Basic-PDE2}
	\left\{ \begin{array}{ll}
	-{\rm div}\, \mathcal{A}(x,\nabla u) = \sigma, \quad u\geq 0  \quad \text{in } \RR^n, \\
	\displaystyle{\lim_{|x|\rightarrow \infty}}\,  u = 0, 
	\end{array}
	\right.
	\end{equation}
where $\sigma \in \MM(\RR^n)$, and 
$\sigma<<{\rm cap}_p$.  	Then $u$ is the unique $\mathcal{A}$-superharmonic solution of \eqref{Basic-PDE2}, which coincides with the minimal 
$\mathcal{A}$-superharmonic reachable solution of \eqref{Basic-PDE}.
\end{theorem}

\begin{proof} First notice that the condition  $\displaystyle{\lim_{|x|\rightarrow \infty}}\,  u = 0$ in \eqref{Basic-PDE2} yields, in view of  \eqref{Wolffbound},  
	$$\lim_{|x|\rightarrow\infty} {\bf W}_{1,p}\sigma (x)=0.$$
	
	For any $\epsilon>0$, let 
	$$\Omega_\ep:=\{x\in\RR^n: u(x)>\ep\},$$
	and 
	$$u_\epsilon:= \max\{u, \epsilon\}-\ep. $$
	
	Clearly, $\Om_\ep$ is a bounded open set,  $u_\ep=u-\epsilon$ on $\Omega_\ep$, and $u_\ep= 0$ in $\RR^n\setminus \Om_\ep$.
	
	Let $v$ be the minimal solution of \eqref{Basic-PDE}, which is also the minimal solution of \eqref{Basic-PDE2}, since $v \le K \, {\bf W}_{1,p}\sigma$, and hence $\displaystyle{\lim_{|x|\rightarrow \infty}}\,  v = 0$. It is enough to show that 
	\begin{equation}\label{uepv}
	u_\ep\leq v 
	\end{equation}
in $\Omega_\ep$, as this will yield that $u\leq v$ in $\RR^n$ after letting $\ep\rightarrow 0^{+}$.

Now by Lemma \ref{CP-corrected} below, to verify \eqref{uepv}, it suffices to show that $u_\ep$
is a renormalized solution of 
\begin{equation*}
\left\{ \begin{array}{ll}
-{\rm div}\, \mathcal{A}(x,\nabla u_\ep) = \sigma   \quad \text{in } \Omega_\ep, \\
u_\ep = 0 \quad \text{on } \partial\Om_\ep. 
\end{array}
\right.
\end{equation*}

Note that, for any $k>0$, $T_k(u_\ep)=T_{k+\epsilon}(u)-\ep$. We have $T_k(u_\ep)\in W^{1,p}_{\rm loc} (\RR^n)$, $T_k(u_\ep)$ is quasi-continuous in $\RR^n$, and $T_k(u_\ep)=0$ everywhere in $\RR^n\setminus\Om_\ep$. Thus $T_k(u_\ep)\in W^{1,p}_0(\Om_\ep)$ (see \cite[Theorem 4.5]{HKM}).

As $u$ is a local renormalized solution in $\RR^n$, 
for every $k>0$ there exists a nonnegative measure $\lambda_{k+\ep}<<{\rm cap}_p$, concentrated on the sets $\{u=k+\ep\}$,  such that
$\lambda_{k+\ep}\rightarrow 0$ weakly  as measures in $\RR^n$ as $k\rightarrow\infty$. Since $\Om_\ep$ is bounded, this implies that $\lambda_{k+\ep}\rightarrow 0$  in the narrow topology of measures in $\Omega_\ep$.

Moreover, for $k>0$,
\begin{equation*}
\int_{\{u<k+\ep\}}\mathcal{A}(x, Du)\cdot\nabla\varphi dx= \int_{\{u<k+\ep\}}\varphi
d\sigma_{0} +  \int_{\RR^n}\varphi d\lambda_{k+\ep}, 
\end{equation*}
for every $\varphi\in{\rm W}_{0}^{1,\,p}(\RR^n)\cap L^{\infty}(\RR^n)$ with compact support in $\RR^n$.
In particular, we have 
\begin{equation*}
\int_{\{u_\ep<k\}\cap \Om_\ep}\mathcal{A}(x, Du)\cdot\nabla\varphi dx= \int_{\{u_\ep<k\}\cap \Om_\ep}\varphi
d\sigma_{0} +  \int_{\{u_\ep=k\}\cap \Om_\ep}\varphi d\lambda_{k+\ep} 
\end{equation*}
for every $\varphi\in{\rm W}_{0}^{1,\,p}(\Om_\ep)\cap L^{\infty}(\Om_\ep)$.

Thus, we conclude that $u_\ep$ is a renormalized solution in $\Om_\ep$, as desired.
\end{proof}	


\section{Quasilinear equations with a sub-natural growth term in $\RR^n$}
\label{sub-natural}

In this section, we study solutions to the equation 
\begin{equation}\label{sub-quasi-pde-a}
	\left\{ \begin{array}{ll}
	-{\rm div}\, \mathcal{A}(x,\nabla u) = \sigma u^q + \mu, \quad u\ge 0   \quad \text{in } \RR^n, \\
	\displaystyle{\liminf_{|x|\rightarrow \infty}}\,  u = 0, 
	\end{array}
	\right.
	\end{equation}
	in the sub-natural growth  case $0<q<p-1$, with $\mu, \sigma \in \MM(\RR^n)$. 
	
	We consider nontrivial $\mathcal{A}$-superharmonic solutions to \eqref{sub-quasi-pde-a} 
	such that $0<u<\infty$ $d \sigma$-a.e., which implies $u \in L^q_{{\rm loc}} (\RR^n, \sigma)$, so that 
	 $\sigma u^q + \mu \in \MM(\RR^n)$ (see \cite{Ver1}). 
		
	As was noted in the Introduction,  $\sigma << {\rm cap}_p$ whenever there exists a nontrivial 
	solution $u$ to \eqref{sub-quasi-pde-a}, for any $\mu$ (in particular, $\mu=0$). 
	
	The existence and uniqueness of nontrivial \textit{reachable} $\mathcal{A}$-superharmonic solutions to \eqref{sub-quasi-pde-a}, under the additional assumption $\mu<< {\rm cap}_p$,  are proved below. 
	Without this restriction on $\mu$, the existence of nontrivial solutions, not necessarily reachable,  was obtained recently in \cite{Ver1}, along with bilateral pointwise estimates of solutions in terms of nonlinear potentials.

	We use this opportunity to make a correction in the proof of the existence 
	property for \eqref{sub-quasi-pde-a} in the case $\mu=0$ given in \cite[Theorem 1.1]{CV2}, which used a version of the comparison principle (\cite{CV2}, Lemma 5.2).  
	It  was invoked in the proof of \cite[Theorem 1.1]{Ver1} as well. 	Some inaccuracies 
	in the statement of this comparison principle and its proof are fixed in the following lemma. 
	The rest of the proofs of \cite[Theorem 1.1]{CV2} and \cite[Theorem 1.1]{Ver1} remains valid with this correction. (See the proof of Theorem \ref{exist} below.) 

\begin{lemma}\label{CP-corrected}
Let $\Om$ be a bounded open set in $\RR^n$. Suppose that $\mu, \nu  \in \MM_b(\Om)$, where $\mu\leq \nu$ and  $\mu<<{\rm cap}_p$. If $u\ge 0$ is a renormalized solution of 
\begin{equation}\label{muequ}
\left\{ \begin{array}{ll}
-{\rm div}\, \mathcal{A}(x,\nabla u) = \mu   \quad \text{in } \, \Omega, \\
u = 0 \quad \text{on } \, \partial\Om, 
\end{array}
\right.
\end{equation}
and if $v\ge 0$ is an $\mathcal{A}$-superharmonic function in $\Om$ with Riesz measure $\nu$ such that $\min\{v,k\}\in W^{1,p}(\Om)$ for any $k>0$, then 
$u\leq v$ a.e. 
\end{lemma}

\begin{proof} Let $\nu_j$, $j>0$, be the Riesz measure of $\min\{v,j\}$. Since $\min\{v,j\}\in W^{1,p}(\Om)$ we see that $\nu_j$ belongs to  the dual of $W^{1, p}_0(\Om)$ (see \cite[Theorem 21.6]{HKM}).
As in \eqref{truncatemeasure}, we have 
\begin{equation*}
\nu_j=\nu|_{\{v<j\}} + \alpha_j 
\end{equation*}
for a measure $\alpha_j\in \MM_b(\Om)$ concentrated in the set $\{v=j\}$. Thus the measure $\mu_j:=\mu|_{\{v<j\}}\leq \nu|_{\{v<j\}} \leq \nu_j$ for any $j>0$.	
This implies that $\mu_j$ also belongs to  the dual of $W^{1,p}_0(\Om)$, and hence there  exists a unique solution $u_j$  to  the equation
\begin{equation}\label{ujequa}
-{\rm div}\, {\mathcal A}(x, \nabla u_j)=\mu_j, \qquad u_j\in W^{1,p}_0(\Om).
\end{equation}
Then by the comparison principle (see \cite[Lemma 5.1]{CV2}) we find
$$0\leq u_1\leq u_2 \leq \dots \leq u_j\leq \min\{v,j\}$$
for any integer $j>0$.  Thus there is a function $\tilde{u}$ on $\Om$ such that $0\leq \tilde{u}\leq v$ a.e. and $u_j\rightarrow \tilde{u}$ as $j\rightarrow \infty$. We now claim that $\tilde{u}$
is also a renormalized solution to  equation \eqref{muequ}. If this is verified then, as $\mu<<{\rm cap}_p$, we must have that $\tilde{u}=u$ a.e. (see \cite{MP, DMOP}) and thus $u\leq v$ a.e. as desired. 

To show that $\tilde{u}$
is  the renormalized solution of \eqref{muequ},  we first use $T_k^{+}(u_j)$, $k>0$, as a test function for \eqref{ujequa} to obtain 
\begin{equation}\label{uboundp}
\alpha \int_{\Om} |\nabla T_k^{+}(u_j)|^p dx \leq k\mu_j(\Om)\leq k\mu(\Om).
\end{equation}

Since $T_k^{+}(u_j)\rightarrow T_k^{+}(\tilde{u})$ a.e. as $j\rightarrow\infty$, we see that $T_k^{+}(\tilde{u})\in W^{1,p}_0(\Om)$ and 
\begin{equation*}
\alpha \int_{\Om} |\nabla T_k^{+}(\tilde{u})|^p dx \leq k\mu(\Om)
\end{equation*}
for any $k>0$. By \cite[Lemmas 4.1 and 4.2]{BB6}, this yields
$$\tilde{u}\in L^{\frac{n(p-1)}{n-p},\infty}(\Om) \quad \text{and} \quad  D\tilde{u}\in L^{\frac{n(p-1)}{n-1},\infty}(\Om). $$
 
 Moreover, arguing as in Step 4 of the proof of Theorem 3.4 in \cite{DMOP}, we see that $\{\nabla u_j\}$ is a Cauchy sequence in measure which converges  to $D\tilde{u}$ a.e. in $\Om$. There is no need to take a subsequence here as the limit is independent of any subsequence. 
 
Moreover, for any Lipschitz function $h: \RR \rightarrow \RR$ such that $h'$ has compact support  and for any function $\varphi\in W^{1,r}(\Om)\cap L^\infty(\Om)$, $r>n$, such that 
$h(\tilde{u})\varphi\in W^{1,p}_0(\Om)$, we have   
$$\int_\Om \mathcal{A}(x, \nabla u_j)\cdot D\tilde{u} \, h'(\tilde{u}) \, \varphi dx + \int_\Om \mathcal{A}(x, \nabla u_j) \cdot  \nabla \varphi \,  h(\tilde{u})dx =\int_\Om h(\tilde{u}) \, \varphi d\mu_j.$$

Thus if the support of $h'$ is in $[-M,M]$, $M>0$, then, using $0\leq u_j\leq \tilde{u}$, we can rewrite the above equality as 
\begin{align*}
\int_\Om \mathcal{A}(x, \nabla T_M^{+}(u_j)) &\cdot \nabla T_M^{+}(\tilde{u}) \, h'(\tilde{u})\,  \varphi dx + \int_\Om \mathcal{A}(x, \nabla u_j) \cdot  \nabla \varphi \,  h(\tilde{u})dx\\
&\qquad  =\int_\Om h(\tilde{u}) \, \varphi d\mu_j=\int_{\{0\leq v<j\}} h(\tilde{u}) \, \varphi d\mu.
\end{align*}

Note that by \eqref{uboundp} and \cite[Lemma 4.2]{BB6}, we have that $\nabla u_j$ is uniformly bounded in 
$ L^{\frac{n(p-1)}{n-1},\infty}(\Om)$ and $\nabla T^{+}_{M}(u_j)$ is uniformly bounded in 
$ L^{p}(\Om)$. Thus by the Vitali Convergence Theorem, the left-hand side of the above equality converges to 
\begin{align*}
&\int_\Om \mathcal{A}(x, \nabla T_M^{+}(\tilde{u})) \cdot \nabla T_M^{+}(\tilde{u}) \, h'(\tilde{u}) \, \varphi dx + \int_\Om \mathcal{A}(x, D \tilde{u}) \cdot  \nabla \varphi \,  h(\tilde{u})dx\\
&\qquad  =\int_\Om \mathcal{A}(x, D \tilde{u}) \cdot D\tilde{u} \, h'(\tilde{u}) \,  \varphi dx + \int_\Om \mathcal{A}(x, D \tilde{u}) \cdot  \nabla \varphi  \, h(\tilde{u})dx.
\end{align*}  

On the other hand, by the Lebesgue Dominated Convergence Theorem we have 
$$\lim_{j\rightarrow\infty}\int_{\{0\leq v<j\}} h(\tilde{u}) \, \varphi d\mu=\int_{\Om} h(\tilde{u}) \, \varphi d\mu.$$

Thus, we get 
\begin{align*}
\int_\Om \mathcal{A}(x, D\tilde{u}) \cdot D\tilde{u} \, h'(\tilde{u}) \, \varphi dx + \int_\Om \mathcal{A}(x, D \tilde{u}) \cdot  \nabla \varphi  \, h(\tilde{u})dx=\int_{\Om} h(\tilde{u}) \, \varphi d\mu,
\end{align*}
which yields that $\tilde{u}$ is the renormalized solution of \eqref{muequ} (see Definition \ref{hprimesupp}).
\end{proof}

We recall that by $\varkappa=\varkappa (\RR^n)$ we denote the least constant in the weighted norm inequality (see \cite{CV2}, \cite{Ver1}) 
\begin{equation}\label{varkappa}
\left(\int_{\RR^n} |\varphi|^q \, d \sigma\right)^{\frac 1 q} \le \varkappa \,  \Vert - \textrm{div}\, \mathcal{A}(x,  \nabla \varphi)\Vert^{\frac{1}{p-1}}_{\M^+(\RR^n)},  
\end{equation} 
for all $\mathcal{A}$-superharmonic  functions $\varphi\ge 0$ in $\RR^n$ such that  $\displaystyle{\liminf_{|x| \to \infty}} \, \varphi(x)=0$. Notice that by estimates \eqref{Wolffbound}, $K^{-1} \, \varphi \le {\bf W}_{1, p} \mu \le K \, \varphi$, where 
$\mu=-\textrm{div}\, \mathcal{A}(x,  \nabla \varphi)\in \M^{+}(\RR^n)$. Here we may  assume without loss of generality that  $\mu \in \M_b^{+}(\RR^n)$, so that 
 ${\bf W}_{1, p} \mu\not\equiv \infty$.  
Consequently,  
\eqref{varkappa} is equivalent to the inequality 
\begin{equation*}
\left(\int_{\RR^n} ({\bf W}_{1, p} \mu)^q \, d \sigma\right)^{\frac 1 q} \le \kappa \,  \Vert \mu \Vert^{\frac{1}{p-1}}_{\M^+(\RR^n)}  \qquad \textrm{for all} \, \, \mu \in \M_b^+(\RR^n), 
\end{equation*} 
where $K^{-1} \, \varkappa\le \kappa \le K \, \varkappa$. In particular, one can replace $\textrm{div}\, \mathcal{A}(x,  \nabla \varphi)$ in \eqref{varkappa}
by  $\Delta_p$,  up to a constant which depends only on $K$.

By $\varkappa(B)$, where  $B$ is a ball in $\RR^n$, we denote the least constant in a similar localized weighted norm inequality 
with the measure $\sigma_B$ in place of $\sigma$, where  
$\sigma_B=\chi_B \, \sigma$ is the restriction of $\sigma$ to $B$.

The so-called \textit{intrinsic} nonlinear potential $\mathbf{K}_{p, q} \sigma$, introduced in \cite{CV2},   
 is defined  by 
\begin{equation*}
\mathbf{K}_{p, q}  \sigma (x)  =  \int_0^{\infty} \left(\frac{ \varkappa(B(x, t))^{\frac{q(p-1)}{p-1-q}}}{t^{n- p}}\right)^{\frac{1}{p-1}}\frac{dt}{t}, \qquad x \in \RR^n.
\end{equation*} 
Here  $B=B(x, t)$ is a ball in $\RR^n$ of radius $t>0$ centered at $x$. 
As was noticed in \cite{CV2},  $\mathbf{K}_{p, q}  \sigma \not\equiv + \infty$ if and only if 
\begin{equation}\label{suffcond1}
\int_1^{\infty} \left(\frac{ \varkappa(B(0, t))^{\frac{q(p-1)}{p-1-q}}}{t^{n- p}}\right)^{\frac{1}{p-1}}\frac{dt}{t} < \infty.  
\end{equation}

By  \cite[Theorem 1.1]{CV2}, there exists a nontrivial $\mathcal{A}$-superharmonic solution to the homogeneous equation \eqref{sub-quasi-pde-a} 
in the case $\mu=0$ if and only if $\mathbf{W}_{1, p}  \sigma \not\equiv + \infty$ and $\mathbf{K}_{p, q}  \sigma \not\equiv + \infty$, i.e., 
conditions \eqref{finiteness} and \eqref{suffcond1} hold. The next theorem shows 
that this solution is actually reachable.

\begin{theorem}\label{exist} Let $0<q<p-1$, and let $\sigma\in \MM(\RR^n)$.  
	Then the nontrivial minimal $\mathcal{A}$-superharmonic solution $u$ of 
	\begin{equation}\label{sub-quasi-pde}
	\left\{ \begin{array}{ll}
	-{\rm div}\, \mathcal{A}(x,\nabla u) = \sigma u^q, \quad u\ge 0   \quad \text{in } \RR^n, \\
	\displaystyle{\liminf_{|x|\rightarrow \infty}}\,  u = 0, 
	\end{array}
	\right.
	\end{equation}
 constructed in the proof of \cite[Theorem 1.1]{CV2} under the conditions \eqref{finiteness} and \eqref{suffcond1}, is an $\mathcal{A}$-superharmonic reachable solution. 
\end{theorem}

\begin{proof} We start with the same construction as in the proof of \cite[Theorem 1.1]{CV2} for the  minimal $\mathcal{A}$-superharmonic solution $u$, but  with datum $\sigma|_{B_m(0)}$ in place of $\sigma$ ($m=1,2, \dots$). 

For a fixed $m$, let $v_m$ be the minimal $\mathcal{A}$-superharmonic solution to the equation 
	\begin{equation*}
	\left\{ \begin{array}{ll}
	-{\rm div}\, \mathcal{A}(x,\nabla v_m) = \sigma|_{B_m(0)} v_{m}^q, \quad v_m\geq 0,   \quad \text{in } \RR^n, \\
	\displaystyle{\liminf_{|x|\rightarrow \infty}}\,  v_m = 0, 
	\end{array}
	\right.
	\end{equation*}
We recall from the construction in \cite{CV2} that 
$$v_m=\lim_{j\rightarrow\infty}(\lim_{k\rightarrow\infty} v_{j, m}^{k}),$$
where $v_{1,m}^k$ ($k=0,1,2,\dots$) is the
$\mathcal{A}$-superharmonic renormalized solution of 
\begin{equation*}
\left\{ \begin{array}{ll}
-{\rm div}\, \mathcal{A}(x,\nabla v_{1,m}^k) = \sigma|_{B_m(0)\cap B_{2^k}(0)} \, w_{0,m}^q   \quad \text{in } B_{2^k}(0), \\
v_{1,m}^k = 0  \quad \text{on } \partial B_{2^k}(0),
\end{array}
\right.
\end{equation*}
with $w_{0,m}=c_0 ({\bf W}_{1,p}(\sigma|_{B_m(0)}))^{\frac{p-1}{p-1-q}}$, and 
$v_{j,m}^k$ ($k=0,1,2,\dots$, $j=2,3, \dots$) is the
$\mathcal{A}$-superharmonic renormalized solution of 
\begin{equation*}
\left\{ \begin{array}{ll}
-{\rm div}\, \mathcal{A}(x,\nabla v_{j,m}^k) = \sigma|_{B_m(0)\cap B_{2^k}(0)}\, (\displaystyle{\lim_{i\rightarrow\infty}} v_{j-1,m}^i)^q   \quad \text{in } B_{2^k}(0), \\
v_{j,m}^k = 0  \quad \text{on } \partial B_{2^k}(0).
\end{array}
\right.
\end{equation*}
Here $c_0$ is a fixed constant such that 
\begin{equation}\label{const-corr} 
0<c_0\leq \min\left\{ \left( \mathfrak{c}^{\frac{q}{p-1-q}} K^{-1} \right)^{\frac{p-1}{p-1-q}}, C K^{\frac{1-p}{p-1-q}}\right \},
\end{equation}
where $C$ is the constant in (3.9) of \cite{CV2}, and $\mathfrak{c}$ is the constant in (3.10) of \cite{CV2} with $\alpha=1$.

We also recall from  \cite{CV2}  that 
$$u=\lim_{j\rightarrow\infty}(\lim_{k\rightarrow\infty} u_{j}^{k}),$$
where $u_{j}^{k}$ are the
$\mathcal{A}$-superharmonic renormalized solutions of the corresponding problems  in $B_{2^k}(0)$ 
with $\sigma$ in place of $\sigma|_{B_m(0)}$.  In particular, $ \min (u_{j}^k, l) \in W^{1, p}_0 (B_{2^k}(0))$ and 
   $ \min (v_{j,m}^k, l) \in W^{1, p}_0 (B_{2^k}(0))$ 
for all $l>0$.

Thus, by the above version of the comparison principle  (Lemma \ref{CP-corrected}) we see that 
$$v_{j,m_1}^k\leq v_{j,m_2}^k \leq u_{j}^k\quad  {\rm in} \, \, B_{2^k}(0),$$
whenever $m_1 \le m_2$. 

This yields 
$$0\leq v_1\leq v_2 \leq \dots \leq v_m\leq u \quad  {\rm in} \, \, \RR^n.$$

Letting now $m\rightarrow\infty$, we obtain an $\mathcal{A}$-superharmonic reachable solution $$v\defeq \lim_{m\to \infty} v_m$$ to
\eqref{sub-quasi-pde} such that $v\leq u$ in $\RR^n$. As $u$ is the minimal $\mathcal{A}$-superharmonic solution  of 
\eqref{sub-quasi-pde}, we see on the other hand that $u\leq v$, and thus $u=v$, which completes the proof. 
\end{proof}

\noindent {\bf Remark.} In the proof of \cite[Theorem 1.1]{CV2}, there is a misprint in the exponent 
in inequality  \eqref{const-corr} above for 
the constant $c_0$. 
This choice of $c_0$ ensures  the minimality of the solution $u$ of \eqref{sub-quasi-pde} 
constructed in \cite{CV2}. 
\smallskip

We recall that, by \cite[Theorem 1.1]{Ver1} and \cite[Remark 4.3]{Ver1}, a nontrivial  $\mathcal{A}$-superharmonic solution of \eqref{sub-quasi-pde-a}  exists 
if and only if $\mathbf{W}_{1, p} \sigma\not \equiv \infty$, $\mathbf{K}_{p, q} \sigma\not \equiv \infty$, and 
 $\mathbf{W}_{1, p} \mu \not \equiv \infty$, i.e., the following three conditions hold:
\begin{align}\label{char1}
& \int_{1}^{\infty} \left(\frac{\sigma(B(0,\rho))}{\rho^{n-p}} \right)^{\frac{1}{p-1}}\frac{d\rho}{\rho}<+\infty,\\ 
& \int_1^{\infty} \left(\frac{ \varkappa(B(0, \rho))^{\frac{q(p-1)}{p-1-q}}}{t^{n- p}}\right)^{\frac{1}{p-1}}\frac{d\rho}{\rho} < \infty, \label{char2}\\   & \int_{1}^{\infty} \left(\frac{\mu(B(0,\rho))}{\rho^{n-p}} \right)^{\frac{1}{p-1}}\frac{d\rho}{\rho}<+\infty.\label{char3}
\end{align}

\begin{theorem}\label{EX} Let $0<q<p-1$, and let $\mu, \sigma\in \MM(\RR^n)$, where $\mu<<{\rm cap}_p$.  
	Then, under the conditions \eqref{char1}, \eqref{char2}, and \eqref{char3}, 
	there exists a  nontrivial minimal reachable $\mathcal{A}$-superharmonic solution of \eqref{sub-quasi-pde-a}.  
\end{theorem}

\begin{proof} Since the case $\mu=0$ was treated in Theorem \ref{exist} above, without loss of generality we may assume that $\mu\not=0$. We recall that in the proof of \cite[Theorem 1.1]{Ver1}, a  nontrivial $\mathcal{A}$-superharmonic solution $u$ of \eqref{sub-quasi-pde-a}, was constructed using the following  iteration process.   
We set $u_0=0$, and for $j=0, 1, 2, \ldots$  
 construct the iterations 
\begin{equation}\label{iter}
	\left\{ \begin{array}{ll}
	-{\rm div}\, \mathcal{A}(x, \nabla u_{j+1}) = \sigma u_j^q + \mu   \quad \text{in } \RR^n, \\
	\displaystyle{\liminf_{|x|\rightarrow \infty}}\,  u_{j+1} = 0, 
	\end{array}
	\right.
\end{equation}
where $u_j \in L^q_{{\rm loc}} (\RR^n, d \sigma)$. 
We observe that, for each $j$, the solution $u_{j+1}$ was chosen in \cite{Ver1} 
so that $u_{j}\le u_{j+1}$ ($j=0, 1, 2, \ldots$) 
by a   version of the comparison principle (see \cite[Lemma 3.7 and Lemma 3.9]{PV2}). Then 
$u \defeq \lim_{j \to \infty} u_j$ is a nontrivial $\mathcal{A}$-superharmonic solution  of \eqref{sub-quasi-pde-a}. 

We now modify this argument as follows to obtain a \textit{minimal} nontrivial $\mathcal{A}$-superharmonic solution of \eqref{sub-quasi-pde-a}. Notice that  $\mu<<{\rm cap}_p$ by assumption, and, as mentioned above, $\sigma<<{\rm cap}_p$, 
since 
a solution exists. Hence, clearly the measure $\sigma u_j^q + \mu <<{\rm cap}_p$ as well. By Theorem \ref{miniexist},  
$u_{j+1}$ can be chosen as the \textit{minimal} $\mathcal{A}$-superharmonic solution to \eqref{iter}. 

It follows by induction  that 
 $u_{j}\le u_{j+1}$ ($j=0, 1, 2, \ldots$). Indeed, this is trivial when $j=0$,  and then by the inductive step,  
 $$
 \sigma u_{j-1}^q +\mu \le  \sigma u_{j}^q +\mu, \quad j=1,2, \ldots ,  
 $$
  which is obvious when $j=1$. From this, using Theorem \ref{comprin} we deduce
 $u_{j}\le u_{j+1}$ for all $j=1, 2, \ldots$. 
  
  Similarly, if $\tilde{u}$ is any $\mathcal{A}$-superharmonic solution of \eqref{sub-quasi-pde-a}, then again arguing by induction  and using Theorem \ref{comprin}, we deduce that $u_{j+1} \le \tilde{u}$ ($j=0, 1, 2, \ldots$), since 
  $$
 \sigma u_{j-1}^q +\mu \le  \sigma  \tilde{u}^q +\mu, \quad j=1,2, \ldots .  
 $$ 
Consequently, $u\le  \tilde{u}$, i.e., $u$ is the minimal $\mathcal{A}$-superharmonic solution of \eqref{sub-quasi-pde-a}.

We next show that $u$ is a reachable solution. Using a similar iteration process with $\sigma|_{B_m(0)}$ in place of $\sigma$ and $\mu|_{B_m(0)}$ in place of $\mu$ ($m=1,2, \dots$), we set $v_{0, m}=0$ and define 
$v_{j, m}$ to 
be the minimal $\mathcal{A}$-superharmonic solution to the equation 
	\begin{equation*}
	\left\{ \begin{array}{ll}
	-{\rm div}\, \mathcal{A}(x,\nabla v_{j+1, m}) = \sigma|_{B_m(0)} v_{j, m}^q + \mu|_{B_m(0)}, \quad v_{j, m}\geq 0,   \quad \text{in } \RR^n, \\
	\displaystyle{\liminf_{|x|\rightarrow \infty}}\,  v_{j, m} = 0, 
	\end{array}
	\right.
	\end{equation*}
	where $v_{j, m} \le v_{j+1, m}$ for each $m=1, 2, \ldots$.
	
As above, arguing by induction and using Theorem \ref{comprin},
we deduce 
$$
v_{j, m_1} \le v_{j, m_2} \le u_j, \quad j=1, 2, \ldots,
$$
whenever $m_1\le m_2$. It follows that $v_m \defeq \lim_{j\to \infty} v_{j, m}\le u$ ($m=1, 2, \ldots$)
is an $\mathcal{A}$-superharmonic  solution of the equation 
\begin{equation*}
	\left\{ \begin{array}{ll}
	-{\rm div}\, \mathcal{A}(x, \nabla v_{m}) =  \sigma|_{B_m(0)} v_{m}^q + \mu|_{B_m(0)}, \quad v_{m}\ge 0  \quad \text{in } \RR^n, \\
	\displaystyle{\liminf_{|x|\rightarrow \infty}}\,  v_{m} = 0, 
	\end{array}
	\right.
\end{equation*}
where $v_{m_1} \le v_{m_2} \le u$ if $m_1\le m_2$.

Thus, letting $m\rightarrow\infty$, we obtain an $\mathcal{A}$-superharmonic \textit{reachable} solution 
$v\defeq\lim_{m\to \infty} v_{m}$ to
\eqref{sub-quasi-pde} such that $v\leq u$. Since $u$ is the minimal $\mathcal{A}$-superharmonic solution  of 
\eqref{sub-quasi-pde}, we see that $u=v$, which completes the proof.
\end{proof}	

We now prove the uniqueness property for reachable solutions of \eqref{sub-quasi-pde-a}. 

\begin{theorem}\label{main-thm-q-p} Let $0<q<p-1$, and let $\mu, \sigma\in \MM(\RR^n)$, where $\mu<<{\rm cap}_p$. Suppose $\mathcal{A}$ satisfies conditions \eqref{structure} and \eqref{homogeneity}. 
Then nontrivial $\mathcal{A}$-superharmonic reachable solutions 
of \eqref{sub-quasi-pde-a} are unique. 
\end{theorem}

\begin{proof} Let $u, v$ be two nontrivial $\mathcal{A}$-superharmonic  solutions 
of \eqref{sub-quasi-pde-a} in $\RR^n$. Then by \cite[Theorem 1.1]{Ver1} and \cite[Remark 4.3]{Ver1}, there exists a  constant $C\ge 1$,  
depending only on $p$, $q$ and $n$, such that 
\begin{equation*}
C^{-1} \, u \le v \le C \, u \quad {\rm in} \, \, \RR^n. 
\end{equation*}
Hence, clearly, 
\begin{equation*}
   -{\rm div}\, \mathcal{A}(x, \nabla v)=\sigma v^q   +  \mu  \le C^q  \left( \sigma u^q  +  \mu \right)=-{\rm div}\, \mathcal{A}(x, \nabla (C^{\frac{q}{p-1}}  u)). 
\end{equation*}
Notice that here by definition $u, v \in L^{q}_{{\rm loc}}(\RR^n, \sigma)$. Suppose that 
$v$ is a reachable solution of \eqref{sub-quasi-pde-a} in $\RR^n$. Then by Corollary \ref{cor-1} with 
$ \sigma v^q  +  \mu$ in place of $\sigma$, and $ \tilde{\sigma} = C^q  \left(\sigma u^q  + \mu \right)$, 
it follows that $v \le C^{\frac{q}{p-1}}u$. 

By iterating this argument, we deduce 
\begin{equation*}
 v \le C^{(\frac{q}{p-1})^j}u \quad {\rm in} \, \, \RR^n, \qquad j=1, 2,  \ldots .  
\end{equation*}
Since $0<q<p-1$, letting $j \to \infty$ in the preceding inequality, we obtain $v \le u$ in $\RR^n$. Interchanging the roles of $u$ and $v$, we see that actually $u=v$ in $\RR^n$. 
\end{proof}

\begin{corollary}\label{cor-4.5}
		 Nontrivial $\mathcal{A}$-superharmonic solutions $u$ 
of \eqref{sub-quasi-pde-a} are unique under the assumptions of Theorem \ref{main-thm-q-p}, provided any one of the following conditions holds:

		\noindent {\rm (i)}   $u \in L^q (\RR^n, d \sigma)$ and $ \mu \in \MM_b(\RR^n)$, or equivalently 
		$\varkappa (\RR^n) < \infty$ and $ \mu \in \MM_b(\RR^n)$;
		
		\noindent {\rm (ii)} 	$\lim_{|x| \to \infty} u(x) =0$;
		
		\noindent {\rm (iii)} $|D u| \in L^{p}(\RR^n)$, or $|D u| \in L^{\gamma,\infty}(\RR^n)$ for some $(p-1)n/(n-1)\leq  \gamma < p$.
		\end{corollary}
		
	\begin{proof} Suppose first that {\rm (i)}  holds. By \cite[Theorem 4.4]{CV2},
	$\varkappa (\RR^n) < \infty$ if and only if   there exists 
	a nontrivial $\mathcal{A}$-superharmonic  solution $u \in L^q(\RR^n, d\sigma)$  of \eqref{sub-quasi-pde}.   
	In particular, since by \cite[Theorem 4.1]{Ver1},
	$$u \ge C \, \left[{\bf W}_{1, p} \sigma + \left({\bf K}_{p, q} \sigma \right)^{\frac{p-1}{p-1-q}} \right],$$
	 it follows that 
	$${\bf W}_{1, p} \sigma\in L^{\frac{q(p-1)}{p-1-q}}(\RR^n, d \sigma) \quad  \textrm{and} \quad  
	{\bf K}_{p, q} \sigma\in L^{q}(\RR^n, d \sigma).$$
	
	Next, we denote by  $\varphi$ an $\mathcal{A}$-superharmonic  solution to the equation 
	\begin{equation*}
\left\{ \begin{array}{ll}
  -{\rm div}\, \mathcal{A}(x, \nabla \varphi) = \mu, \quad \varphi \geq 0  \quad \text{in } \RR^n, \\
\displaystyle{\liminf_{|x|\rightarrow \infty}}\,  \varphi = 0, 
\end{array}
\right.
\end{equation*}
where $\mu$ is the Riesz measure of $\varphi$. Notice that $\varphi \ge K^{-1} 	\, {\bf W}_{1, p} \mu$ by the lower bound in inequality \eqref{Wolffbound}.
	  Since $ \mu \in \M_b(\RR^n)$ and $\varkappa (\RR^n) < \infty$, using  $\varphi$ as a test function in  inequality \eqref{varkappa}  yields ${\bf W}_{1, p} \mu \in L^q (\RR^n, d \sigma)$.

	 Hence, by \cite[Theorem 1.1]{Ver1} and \cite[Remark 4.3]{Ver1}, we deduce 
	 that there exists a nontrivial $\mathcal{A}$-superharmonic solution of \eqref{sub-quasi-pde-a} $u \in  L^q (\RR^n, d \sigma)$, and, for any such a solution, $\sigma u^q  + \mu\in \M_b(\RR^n)$. 
	 It follows that   $u$ is a reachable $\mathcal{A}$-superharmonic solution 
	of \eqref{sub-quasi-pde-a} by Theorem \ref{EandU} and Theorem \ref{us} (iii).

	In case {\rm (ii)},  by Theorem \ref{thm-lim} $u$ is a reachable solution 
	of \eqref{sub-quasi-pde-a}. 
	
	In case  {\rm (iii)}, $u$ is a  reachable $\mathcal{A}$-superharmonic solution 
	of \eqref{sub-quasi-pde-a} by Theorem \ref{EandU} and Theorem \ref{us} (i), (ii). 
	
	In all these cases, reachable $\mathcal{A}$-superharmonic solutions are unique by 
	Theorem \ref{main-thm-q-p}. 
	\end{proof}	
	
	\begin{remark}  Uniqueness of \textit{finite energy} solutions $u$ 
	to \eqref{sub-quasi-pde-a} such that   $|D u| \in L^{p}(\RR^n)$ in Corollary \ref{cor-4.5}(iii)  
	was established  in \cite[Theorem 6.1]{SV}  in the special case of the $p$-Laplace operator 
	using a different method. (See also an earlier result \cite[Theorem 5.1]{CV1} 
	in the case $\mu=0$.) Solutions of finite energy  to \eqref{sub-quasi-pde-a} exist if and only if ${\bf W}_{1, p} \sigma\in L^{\frac{(1+q)(p-1)}{p-1-q}}(\RR^n, d \sigma)$ and ${\bf W}_{1, p} \mu \in L^{1}(\RR^n, d \mu)$   
	(\cite[Theorem 1.1]{SV}).  
	\end{remark} 
	

	\begin{remark}\label{rem-last}  Under the assumptions of  Theorem \ref{EX}, 
	but without the restriction $\mu<<{\rm cap}_p$, it is still possible to prove the existence of an $\mathcal{A}$-superharmonic reachable solution
		(not necessarily minimal) of \eqref{sub-quasi-pde-a}. The construction 
		of such a solution makes use of  an  extension of \cite[Lemma 6.9]{PV1} proved below.
\end{remark}

	\begin{proof} 	To prove this claim, we shall construct first a nondecreasing sequence $\{u_m\}_{m\geq 1}$ of  $\mathcal{A}$-superharmonic solutions of 
		\begin{equation*}
		\left\{ \begin{array}{ll}
		-{\rm div}\, \mathcal{A}(x,\nabla u_{m}) = \sigma|_{B_m(0)} u_{m}^q + \mu|_{B_m(0)}, \quad u_{m}\geq 0,   \quad \text{in } \RR^n, \\
		\displaystyle{\liminf_{|x|\rightarrow \infty}}\,  u_{m} = 0. 
		\end{array}
		\right.
		\end{equation*}
Then by    \cite[Theorem 1.1 and Remark 4.3]{Ver1}, 
$$u_m(x) \leq C\left( {\bf W}_{1,p} \mu(x) + {\bf K}_{p,q}\sigma(x) +  [{\bf W}_{1,p} \sigma(x)]^{\frac{p-1}{p-1-q}} \right), \quad x\in\RR^n. $$
It follows from \cite[Theorem 1.17]{KM1} that $u_{m}\rightarrow u$ pointwise everywhere, 
where $u$ is   an $\mathcal{A}$-superharmonic reachable solution   of \eqref{sub-quasi-pde-a}.

		The construction of $\{u_m\}_{m\geq 1}$ can be done as follows. 
		 It suffices  to demonstrate only how to construct $u_1$ and $u_2$ such that $u_2\geq u_1$, since the construction of $u_m$ for $m\ge 3$ is completely analogous. 
		 Let $v_1$ be an  $\mathcal{A}$-superharmonic solution of 
		\begin{equation*}
		\left\{ \begin{array}{ll}
		-{\rm div}\, \mathcal{A}(x,\nabla v_{1}) = \mu|_{B_1(0)}, \quad v_{1}\geq 0,   \quad \text{in } \RR^n, \\
		\displaystyle{\liminf_{|x|\rightarrow \infty}}\,  v_{1} = 0. 
		\end{array}
		\right.
		\end{equation*}
		Here as above $v_1$ is an a.e. pointwise limit of a subsequence of $\{v_1^{(k)}\}_{k\geq 1}$, where each  $v_1^{(k)}$ is a nonnegative $\mathcal{A}$-superharmonic renormalized solution of 
		\begin{equation*}
		\left\{ \begin{array}{ll}
		-{\rm div}\, \mathcal{A}(x,\nabla v_{1}^{(k)}) = \mu|_{B_1(0)}   \quad \text{in } B_k(0), \\
		v_{1}^{(k)} = 0 \quad \text{on } \partial B_k(0). 
		\end{array}
		\right.
		\end{equation*}
		Next,  for any $j\geq1$, let $v_{j+1}$ be an  $\mathcal{A}$-superharmonic solution of 
		\begin{equation*}
		\left\{ \begin{array}{ll}
		-{\rm div}\, \mathcal{A}(x,\nabla v_{j+1}) = \sigma|_{B_1(0)} v_{j}^q + \mu|_{B_1(0)}, \quad v_{j+1}\geq 0,   \quad \text{in } \RR^n, \\
		\displaystyle{\liminf_{|x|\rightarrow \infty}}\,  v_{j+1} = 0.  
		\end{array}
		\right.
		\end{equation*}
	Notice that $v_{j+1}$ is an a.e. pointwise limit of a subsequence of $\{v_{j+1}^{(k)}\}_{k\geq 1}$, where each  $v_{j+1}^{(k)}$ is a nonnegative $\mathcal{A}$-superharmonic renormalized solution of 
		\begin{equation*}
		\left\{ \begin{array}{ll}
		-{\rm div}\, \mathcal{A}(x,\nabla v_{j+1}^{(k)}) =\sigma|_{B_1(0)} v_{j}^q  +\mu|_{B_1(0)}   \quad \text{in } B_k(0), \\
		v_{j+1}^{(k)} = 0 \quad \text{on } \partial B_k(0). 
		\end{array}
		\right.
		\end{equation*}
		By \cite[Lemma 6.9]{PV1} we may assume that $v_{2}^{(k)}\geq v_1^{(k)}$ for all $k\geq 1$, and hence $v_2\geq v_1$. In the same way, by  induction
		we deduce that $v_{j+1}^{(k)}\geq v_j^{(k)}$ for all $j, k\geq 1$. It follows that $v_{j+1}\geq v_j$, and 
		$$v_{j+1} \leq C\, {\bf W}_{1,p}(\sigma v_{j+1}^q) + C\, {\bf W}_{1,p}(\mu).$$
Then by \cite[Theorem 4.1]{Ver1}, for any $j \geq 1$, we obtain the bound 
\begin{equation}\label{vj1}
v_{j+1}(x) \leq C\left( {\bf W}_{1,p} \mu(x) + {\bf K}_{p,q}\sigma(x) +  [{\bf W}_{1,p} \sigma(x)]^{\frac{p-1}{p-1-q}} \right), \quad x\in\RR^n. 
\end{equation}
Thus, the nondecreasing sequence $\{v_j\}_{j\geq 1}$ converges to an 
		$\mathcal{A}$-superharmonic solution $u_1$ of 
		\begin{equation*}
		\left\{ \begin{array}{ll}
		-{\rm div}\, \mathcal{A}(x,\nabla u_{1}) = \sigma|_{B_1(0)} u_{1}^q + \mu|_{B_1(0)}, \quad u_{1}\geq 0,   \quad \text{in } \RR^n, \\
		\displaystyle{\liminf_{|x|\rightarrow \infty}}\,  u_{1} = 0. 
		\end{array}
		\right.
		\end{equation*}

		To construct $u_2$ such that $u_2\geq u_1$, let $w_1$ be an  $\mathcal{A}$-superharmonic solution of 
		\begin{equation*}
		\left\{ \begin{array}{ll}
		-{\rm div}\, \mathcal{A}(x,\nabla w_{1}) = \mu|_{B_2(0)}, \quad w_{1}\geq 0,   \quad \text{in } \RR^n, \\
		\displaystyle{\liminf_{|x|\rightarrow \infty}}\,  w_{1} = 0. 
		\end{array}
		\right.
		\end{equation*}

Notice that  $w_1$ is an a.e. pointwise limit of a subsequence of $\{w_1^{(k)}\}_{k\geq 1}$, where each  $w_1^{(k)}$ is a nonnegative $\mathcal{A}$-superharmonic renormalized solution of 
		\begin{equation*}
		\left\{ \begin{array}{ll}
		-{\rm div}\, \mathcal{A}(x,\nabla w_{1}^{(k)}) = \mu|_{B_2(0)}   \quad \text{in } B_k(0), \\
		w_{1}^{(k)} = 0 \quad \text{on } \partial B_k(0). 
		\end{array}
		\right.
		\end{equation*}
		Again, by \cite[Lemma 6.9]{PV1} we may assume that $w_{1}^{(k)}\geq v_1^{(k)}$ for all $k\geq 1$, and hence $w_1\geq v_1$.
		
		Next,  for any $j\geq1$, let $w_{j+1}$ be an  $\mathcal{A}$-superharmonic solution of 
		\begin{equation*}
		\left\{ \begin{array}{ll}
		-{\rm div}\, \mathcal{A}(x,\nabla w_{j+1}) = \sigma|_{B_2(0)} w_{j}^q + \mu|_{B_2(0)}, \quad w_{j+1}\geq 0,   \quad \text{in } \RR^n, \\
		\displaystyle{\liminf_{|x|\rightarrow \infty}}\,  w_{j+1} = 0. 
		\end{array}
		\right.
		\end{equation*}
		
		Notice that  $w_{j+1}$ is an a.e. pointwise limit of a subsequence of $\{w_{j+1}^{(k)}\}_{k\geq 1}$, where each  $w_{j+1}^{(k)}$ is a nonnegative $\mathcal{A}$-superharmonic renormalized solution of 
		\begin{equation*}
		\left\{ \begin{array}{ll}
		-{\rm div}\, \mathcal{A}(x,\nabla w_{j+1}^{(k)}) =\sigma|_{B_2(0)} w_{j}^q  +\mu|_{B_2(0)}   \quad \text{in } B_k(0), \\
		w_{j+1}^{(k)} = 0 \quad \text{on } \partial B_k(0). 
		\end{array}
		\right.
		\end{equation*}
	
	We can ensure here that $w_{j+1}^{(k)}\geq \max\{v_{j+1}^{(k)}, w_{j}^{(k)}\}$ for all $j, k\geq 1$. Indeed, 
		since $w_1\geq v_1$ and $w_1\geq 0$, by Lemma \ref{suppl} below we may 
	assume  that  $w_{2}^{(k)}\geq \max\{v_{2}^{(k)}, w_{1}^{(k)}\}$ for all $k\geq 1$, and hence  $w_{2}\geq \max\{v_{2}, w_1\}$.
Repeating this argument by induction we obtain  $w_{j+1}^{(k)}\geq \max\{v_{j+1}^{(k)}, w_{j}^{(k)}\}$ for all $j, k\geq 1$. 

It follows that  $w_{j+1}\geq \max\{v_{j+1}, w_{j}\}$ for all $j\geq 1$. As in \eqref{vj1} we have
\begin{equation*}
w_{j+1}(x) \leq C\left( {\bf W}_{1,p} \mu(x) + {\bf K}_{p,q}\sigma(x) +  [{\bf W}_{1,p} \sigma(x)]^{\frac{p-1}{p-1-q}} \right), \quad x\in\RR^n, 
\end{equation*}
and hence $\{ w_j\}$ is a nondecreasing sequence   which converges to an 
		$\mathcal{A}$-superharmonic solution $u_2$ of 
		\begin{equation*}
		\left\{ \begin{array}{ll}
		-{\rm div}\, \mathcal{A}(x,\nabla u_{2}) = \sigma|_{B_2(0)} u_{2}^q + \mu|_{B_2(0)}, \quad u_{2}\geq 0,   \quad \text{in } \RR^n, \\
		\displaystyle{\liminf_{|x|\rightarrow \infty}}\,  u_{2} = 0 
		\end{array}
		\right.
		\end{equation*} 
such that 	$u_2\geq u_1$, as desired. \end{proof}
	
	 The following lemma invoked in the argument presented above 
	 is an extension of \cite[Lemma 6.9]{PV1}.
	\begin{lemma}\label{suppl}
		Let $\Om$ be a bounded open set in $\RR^n$ and let $\mu_1, \mu_2 \in \mathcal{M}_{b}^+(\Om)$. Suppose that $u_i$ $(i=1, 2)$ is a renormalized solution of 
		\begin{equation*}
		\left\{ \begin{array}{ll}
		-{\rm div}\, \mathcal{A}(x,\nabla u_i) =\mu_{i}   \quad \text{in } \Om, \\
		u_i = 0 \quad \text{on } \partial \Om. 
		\end{array}
		\right.
		\end{equation*}
Then for any measure $\nu \in \mathcal{M}_{b}^+(\Om)$ such that $\nu\geq \mu_1$ and $\nu\geq \mu_2$, there is a renormalized solution $v$ of 
\begin{equation*}
\left\{ \begin{array}{ll}
-{\rm div}\, \mathcal{A}(x,\nabla v) = \nu   \quad \text{in } \Om, \\
v = 0 \quad \text{on } \partial \Om, 
\end{array}
\right.
\end{equation*}
such that $v\geq u_1$ and $v\geq u_2$ a.e.
	\end{lemma}

\begin{proof}
For $i=1, 2$, let $u_{i,k}=\min\{u_i,k\}$ ($k=1,2, \dots$). Then $u_{i,k}$ is the bounded renormalized solution of 
	\begin{equation*}
	\left\{ \begin{array}{ll}
	-{\rm div}\, \mathcal{A}(x,\nabla u_{i,k}) =\mu_{i0}|_{\{u_i<k\}}  + \lambda_{i,k} \quad \text{in } \Om, \\
	u_{i,k} = 0 \quad \text{on } \partial \Om. 
	\end{array}
	\right.
	\end{equation*}
	Here $\mu_i=\mu_{i0}+\mu_{is}$ ($i=1,2$) is the decomposition of $\mu_i$ 
	used in Sec. \ref{Quasilinear} above,   where  $\mu_{i0}, \mu_{is} \in \mathcal{M}_{b}^+(\Omega)$,  
$\mu_{i0}<< {\rm cap}_p$, and $\mu_{is}$ is concentrated on a set of zero $p$-capacity. Moreover, $\lambda_{i,k}\in\mathcal{M}_{b}^+(\Omega)$ and $\lambda_{i,k} \rightarrow \mu_{is}$ in the narrow topology of measures as $k\rightarrow\infty$ (see Definition \ref{defrs}).
	
	Now let $v_k$ ($k=1,2, \dots$) be a renormalized solution of 
	 \begin{equation*}
	 \left\{ \begin{array}{ll}
	 -{\rm div}\, \mathcal{A}(x,\nabla v_{k}) = \sum_{i=1}^2 ( \mu_{i0} + \lambda_{i,k}) + (\nu-\mu_1) +(\nu-\mu_2) \quad \text{in } \Om, \\
	 v_{k} = 0 \quad \text{on } \partial \Om. 
	 \end{array}
	 \right.
	 \end{equation*}
	 
	 Then by \cite[Lemma 6.8]{PV1} we deduce $v_k\geq \max\{u_{1,k}, u_{2,k}\}$ for all $k\geq 1$. Finally, we use the stability results of \cite{DMOP} to find a 
	 subsequence of $\{v_k\}$ that converges a.e. to a desired function $v$.
\end{proof}

\end{document}